\documentclass[10pt,a4paper]{article}
\usepackage[utf8]{inputenc}
\usepackage{amsmath, amsthm}
\usepackage{amsfonts}
\usepackage{amssymb}
\usepackage{bbm}
\usepackage{hyperref}
	
\newcommand{\snabla}{\slashed{\nabla}_{\mathbb{S}^2}}

\usepackage{epigraph}
\newcommand{\sDelta}{\slashed{\Delta}_{\mathbb{S}^2}}
\newcommand{\psiR}{|\psi_0|^2(u,v_R(u))}
\newcommand{\dpsiR}{|\rd_v\psi_0|^2(u,v_R(u))}

\newcommand{\Thetazero}{\Theta_0}
\newcommand{\Cuone}{C_{u_1}}
\newcommand{\boundary}{C_R\int_{u_1}^{u_2}\sum_{|\alpha|\leq 1}|\rd_{\alpha}\phi|^2(u',v_{R}(u')) du'}
\newcommand{\boundaryzero}{C_R\int_{u_1}^{u_2}\sum_{|\alpha|\leq 1}|\rd_{\alpha}\phi_0|^2(u',v_{R}(u')) du'}
\newcommand{\boundarytwo}{C_R\int_{u_1}^{u_2}\sum_{|\alpha|\leq 1,\ |\beta|\leq 1}|\rd_{\alpha}\snabla^{\beta}\phi|^2(u',v_{R}(u')) du'}
\newcommand{\cw}{\check{w}}
\newcommand{\cq}{\check{q}}
\newcommand{\cW}{\check{W}}
\newcommand{\cQ}{\check{Q}}
\newcommand{\tw}{\tilde{w}}
\newcommand{\tq}{\tilde{q}}

\newcommand{\N}{\Sigma}
\newcommand{\D}{\mathcal{D}_R}
\newcommand{\Cutwo}{C_{u_2}}

\newcommand{\DD}{\mathcal{D}_R}

\newcommand{\phione}{\phi_{\geq 1}}
\newcommand{\psione}{\psi_{\geq 1}}

\newcommand{\ep}{\epsilon}

\usepackage{slashed}
\usepackage[left=2cm,right=2cm,top=2cm,bottom=2cm]{geometry}

\usepackage{authblk}
\usepackage{enumerate}

\usepackage{enumerate}

\theoremstyle{plain}
\newtheorem{thm}{Theorem}[section]

\newtheorem{lemma}[thm]{Lemma}

\newtheorem{prop}[thm]{Proposition}

\newtheorem{cor}[thm]{Corollary}

\theoremstyle{remark}

\theoremstyle{definition}

\newcommand{\RR}{\mathbb{R}}

\newcommand{\rd}{\partial}
\newcommand{\ls}{\lesssim}

\theoremstyle{plain}

\theoremstyle{remark}

\theoremstyle{definition}

\usepackage{hyperref, todonotes, mathtools}	

\mathtoolsset{showonlyrefs}

\usepackage{bm}

\usepackage{color}

\usepackage{float}

\addtocounter{tocdepth}{-2}
\usepackage{graphicx}
\usepackage{authblk}
 
\numberwithin{equation}{section}

\title{An extension of the $r^p$ method for wave equations\\ with  scale-critical  potentials and first-order terms}

\author[1]{Maxime~Van~de~Moortel\thanks{maxime.vandemoortel@rutgers.edu}}
\affil[1]{\small  Department of Mathematics, Rutgers University, 
	Hill~Center,~New~Brunswick~NJ~08854,~United~States~of~America \vskip.1pc \  }

\date{\today}

\begin{document}
	\maketitle
	\thispagestyle{empty}
	\thispagestyle{empty}
\begin{center}{\it\large Dedicated to Professor Avraham Soffer, with  friendship and admiration}
\end{center}
		\begin{abstract}  
			
				The $r^p$ method, first introduced in \cite{rp}, has become a robust strategy to prove decay for wave equations in the context of black holes and beyond. In this  note, we propose an extension of this method, which is particularly suitable for proving decay for a general class of wave equations featuring a scale-critical time-dependent potential and/or first-order terms of small amplitude. 
	Our approach consists of absorbing  error terms in the $r^p$-weighted energy using a novel Grönwall argument, which allows a larger range of $p$ than the standard method.
	A spherically symmetric version of our strategy first appeared in \cite{Moi2} in the context of a weakly charged scalar field on a black hole  whose equations also involve a scale-critical potential.
	\end{abstract}

	\section{Introduction} The $r^p$ method of Dafermos--Rodnianski \cite{rp} is a versatile tool  to prove decay in time for solutions $\phi$ to waves equations of the form $\Box_g \phi=0$,  which is sufficiently robust for applications to nonlinear wave equations (see, e.g., \cite{DHRTquasi,DHRTquasi2,FedericoBornInfeld,YangYu}). The method is carried out  in physical space
and	relies on the radiative structure of wave equations on an asymptotically flat spacetime $g$, for which  $r\phi$ (usually) admits a finite limit --  the Friedlander radiation field  $\psi_{\mathcal{I}}(u,\omega)$ -- defined,  for $u\in \RR$, $\omega \in \mathbb{S}^2$ as \begin{equation}
\psi_{\mathcal{I}}(u,\omega):= \lim_{v\rightarrow +\infty} r\phi(u,v,\omega),
\end{equation} 
 where $u$ and $v$ are respectively retarded-time and advanced-time coordinates, corresponding to $u=t-r$ and $v= t+r$ on the usual Minkowski spacetime. In its simplest expression, the key idea of the $r^p$ method is to exploit the boundedness of a $r^p$ weighted energy of the following form, for $0 \leq p \leq 2$: \begin{equation}
\sup_{u}  \sum_{|\beta| \leq 2}\int_{v=v_0}^{+\infty} \int_{\mathbb{S}^2} r^{p} |\snabla^{\beta}\rd_v (r\phi)|^2(u,v,\omega)  dv d\omega,
	\end{equation}  to obtain pointwise decay in  time of $\phi$ at the rate $u^{-\frac{p}{2}}$ as $u\rightarrow +\infty$, under the additional conditions that \begin{enumerate}[I.]
	\item\label{condition1} Energy boundedness (in the style of \eqref{boundedness} below) holds.
	\item \label{condition2} An integrated local decay estimate (in the style of \eqref{ILED}), also known as  a Morawetz estimate,  is valid.
	\end{enumerate}

In this paper, our goal is to apply the $r^p$ method and obtain  decay in time estimates  for the following class of linear wave equations  with scale-critical potential and/or scale-critical first-order terms of small amplitude.	\begin{equation}\label{wave.main}\begin{split}
	&			\Box_g \phi = \frac{1}{r^2(u,v)}  \left( \sum_{i=0}^{1}[\ep w_i(u,v)+ W_i(u,v) ]\cdot  \rd_u^{i}\phi+  [ \ep q(u,v)+ Q(u,v) ]\cdot r \rd_v \phi \right),\\ & 	g= -\Omega^2(u,v) du dv + r^2(u,v) d\sigma_{\mathbb{S}^2}, \end{split}\end{equation}
 where $\ep \in \RR$ is a small constant, and $g$ is a spherically-symmetric and asymptotically flat\footnote{As we will discuss in Section~\ref{section.example}, the usual Minkowski metric $m=-dt^2+ dx^2+dy^2+dz^2$, associated to  $\Box_m  =- \rd_{t}^2+\rd_{x}^2+\rd_{y}^2+\rd_{ z}^2$, satisfies \eqref{H0}, together with many other usual spacetime metrics $g$, such as the Schwarzschild spacetime.} Lorentzian metric in the mild sense that  \begin{equation}\begin{split}\label{H0}
&|1+\rd_u r(u,v)|,\ |1+\rd_v r(u,v)|,\  |\Omega^2(u,v)-4|,\ r|\rd_v\Omega^2|(u,v) \lesssim r^{-1}(u,v) \text{ as } v \rightarrow +\infty,\\ & |\rd_{v}^2 r|(u,v),\ |\rd_{u} \rd_v  r|(u,v),\ |\rd_{u} \rd_{v}^2 r|(u,v)   \lesssim r^{-2}(u,v) \text{ as } v \rightarrow +\infty.\end{split}
 \end{equation}
  Finally, the potentials (terms involving $w_0(u,v)$ and $W_0(u,v)$) and first-order terms (terms involving $w_1(u,v)$, $W_1(u,v)$, $q(u,v)$ and $Q(u,v)$) are scale-critical in the sense that
	\begin{equation}\begin{split}\label{H1}
&  |w_i|(u,v),\ |q|(u,v),\ r|\rd_v w_0|(u,v),\ r^2|\rd_v q|(u,v)    \lesssim 1,\\ & |W_i|(u,v),\ |Q|(u,v),\ r|\rd_v W_0|(u,v),\ r^2|\rd_v W_1|(u,v),\ r^2 |\rd_v Q|(u,v)   = o(1) \text{ as }  v \rightarrow +\infty.
		\end{split}
	\end{equation}
	
	Note that the inverse-square power $r^{-2}$ in the potential of \eqref{wave.main} is called scale-critical because on flat spacetime (i.e., if $g=-dt^2+dx^2+dy^2+dz^2$ the Minkowski metric) and if $w_1=W_1=0$, while $w_0$, $W_0$, $q$, $Q$  are constants, then the rescaling $u \rightarrow \lambda u$,  $v \rightarrow \lambda v$ leaves \eqref{wave.main} unchanged.
The class of potentials satisfying \eqref{H1} are allowed to oscillate in $u$ and mildly oscillate in $v$, such as the following example (assuming \eqref{H0} holds)  \begin{equation}
		w_0(u,v) = \sin(u + \log(r(u,v))),\ w_1(u,v) = q(u,v)=W_0(u,v)= W_1(u,v)= Q(u,v)=0.
	\end{equation}
In particular, we allow for a large class of \emph{time-dependent potentials},	see  Section~\ref{section.example} for more general examples.

In our main theorem below, we take $\ep$ to be sufficiently small, and we conditionally assume that the energy boundedness (condition~\ref{condition1}) and an integrated local decay estimate   (condition~\ref{condition2}) are satisfied, in the traditional spirit of the $r^p$ method \cite{rp}.  We then deduce time-decay of the energy (\eqref{energy.est}) on a foliation  that reaches null infinity (see Figure~\ref{fig}) and point-wise decay (\eqref{pointwise.est}, \eqref{pointwise.est2}) at rates that are arbitrarily close to the optimal ones as $\ep \rightarrow 0$ (see Section~\ref{section.sharp} for further discussions on the sharpness of our estimates).
	
	\begin{thm}\label{main.thm}  Let $\phi$ be a solution of \eqref{wave.main} where $g$ and the potential terms satisfy \eqref{H0} and \eqref{H1} with (characteristic) initial data on the bifurcate null cones $([u_0,u_F] \times \mathbb{S}^2 )\cup( [v_0,+\infty)\times \mathbb{S}^2)$ with $u_0 \in \RR$, $u_F \in \RR \cup \{+\infty\}$, $v_0 \in \RR$ as depicted in Figure~\ref{fig}, satisfying the following assumptions for all $0 \leq p < 3$ \begin{equation}\label{data1}
\sum_{|\beta|\leq 2}\int_{v_0}^{+\infty} \int_{\mathbb{S}^2} r^{p+|\beta|}(u_0,v) |\rd_v (r\snabla^{\beta}\phi)|^2(u_0,v,\omega) dv d\omega,\ \sum_{k=0}^{1}\int_{v_0}^{+\infty}\int_{\mathbb{S}^2} r^{p+2k}(u_0,v)  |\rd_v^{1+k} (r\phi)|^2(u_0,v,\omega) dv d\omega<\infty.
		\end{equation}
Defining $E[\phi](u)$ to be the standard (unweighted) energy  on a foliation radiating to infinity of Figure~\ref{fig}, defined in Section~\ref{prelim.section}, we  assume that for all  solutions $\phi$ as above, the boundedness of the energy $E[\phi](u)$, in the sense that there exists a constant $D>0$ such that all $u_1<u_2$  \begin{equation}\label{boundedness}E[\phi](u_2) \leq D  E[\phi](u_1).\end{equation}	
		Additionally, we make the following assumption of integrated local decay: for some $\sigma>1$, there exists a constant $C_{\sigma}>0$ such that for all $u_1 \geq u_0$, and defining the spacetime region $\mathcal{D}_{u_1,\infty}=\{ u_1\leq u,\ v\geq v_0 \}$  \begin{equation}\label{ILED}
	\int_{\mathcal{D}_{u_1,\infty}}  \frac{|\phi|^2(u,v,\omega)+r^2|\rd_u \phi|^2(u,v,\omega)+
		r^2|\rd_v \phi|^2(u,v,\omega)}{(1+r)^{\sigma}} du dv d\omega  \leq C_{\sigma} E[\phi](u_1),
\end{equation}

  Then, there exists $\ep_0>0$ such that if $|\ep|<\ep_0$, there exists $\eta(\ep)>0$ such that $\eta(\ep) \rightarrow 0 $ as $\ep \rightarrow 0$ and \begin{equation}\label{energy.est}
E[\phi](u)  \lesssim u^{-3+\eta(\ep)}.
\end{equation} The following limit exists $\underset{v\rightarrow+\infty}{\lim} r\phi(u,v,\omega)=\psi_{\mathcal{I}}(u,\omega)$, and the following estimate holds in the region $v-u\geq R$, assuming $R$ is sufficiently large \begin{equation}\label{pointwise.est}
r|\phi|(u,v,\omega),\  |\psi_{\mathcal{I}}|(u,\omega) \lesssim u^{-1+\frac{\eta(\ep)}{2}}.
\end{equation} Moreover,  if  analogues of the energy boundedness \eqref{boundedness} and integrated local  decay  \eqref{ILED}  hold for $T\phi$,  where $T= \rd_u + \rd_v $, in the sense that there exists $\sigma'>1$ and  $\eta'(\ep)>0$ such that $\eta'(\ep) \rightarrow 0 $  as $\ep \rightarrow 0$ such that   for all $u_1<u_2$
\begin{equation}\label{boundedness.T}
E[T\phi](u_2) \leq D' \left( E[T\phi](u_1)	+u_1^{-2+\eta'(\ep)} E[\phi](u_1) \right),
\end{equation}\begin{equation}\label{ILED.T}
\int_{\mathcal{D}_{u_1,\infty}}  \frac{|T\phi|^2(u,v,\omega)}{(1+r)^{\sigma'}} du dv d\omega  \leq C_{\sigma'} \left(E[T\phi](u_1)+u_1^{-2+\eta'(\ep)} E[\phi](u_1) \right),
	\end{equation} for  $C_{\sigma'}>0$, $D'>0$ independent of $\ep$, and additionally the following assumption  is satisfied \begin{equation}\label{H3}
 |\rd_uw_i|(u,v),\ |\rd_uq|(u,v)\ls r^{-1}(u,v)  ,\ r|\rd_u W_i|(u,v),\ r|\rd_u Q|(u,v) =o(1) \text{ as } v\rightarrow+\infty.
	\end{equation}
	 Then, we also have 
	 the following (sharper) pointwise decay estimate: for all $v\geq u+R$, $\omega \in \mathbb{S}^2$:
\begin{equation}\label{pointwise.est2}
	|\phi|(u,v,\omega) \lesssim u^{-2+\frac{\eta(\ep)}{2}}.
\end{equation}
	\end{thm}
	We want to emphasize that, despite initial data being given on  characteristic bifurcate hypersurfaces $([u_0,u_F] \times \mathbb{S}^2 )\cup( [v_0,+\infty)\times \mathbb{S}^2)$, it is equivalent to start from a spacelike hypersurface (Cauchy problem) and evolve the solution up to $([u_0,u_F] \times \mathbb{S}^2 )\cup( [v_0,+\infty)\times \mathbb{S}^2)$  with no difficulty, see e.g.\ \cite{Moi2}, Section 4.

	We finally remark that it is also possible to propagate the point-wise estimates \eqref{pointwise.est}, \eqref{pointwise.est2} of Theorem~\ref{main.thm} in the region $\{ v-u \leq R\}$ (see \cite{AAG0,twotails}). We will however omit the relevant estimates for brevity. On a black hole spacetime,  \eqref{pointwise.est}, \eqref{pointwise.est2} are also valid up to the black hole event horizon, providing so-called red-shift estimates are satisfied, see for instance \cite{Moi2,inverse.Dejan} for such examples and the discussion in Section~\ref{section.example}.
	
		\begin{figure}[H]\label{fig}\begin{center}
	\includegraphics[width=80 mm, height=80 mm]{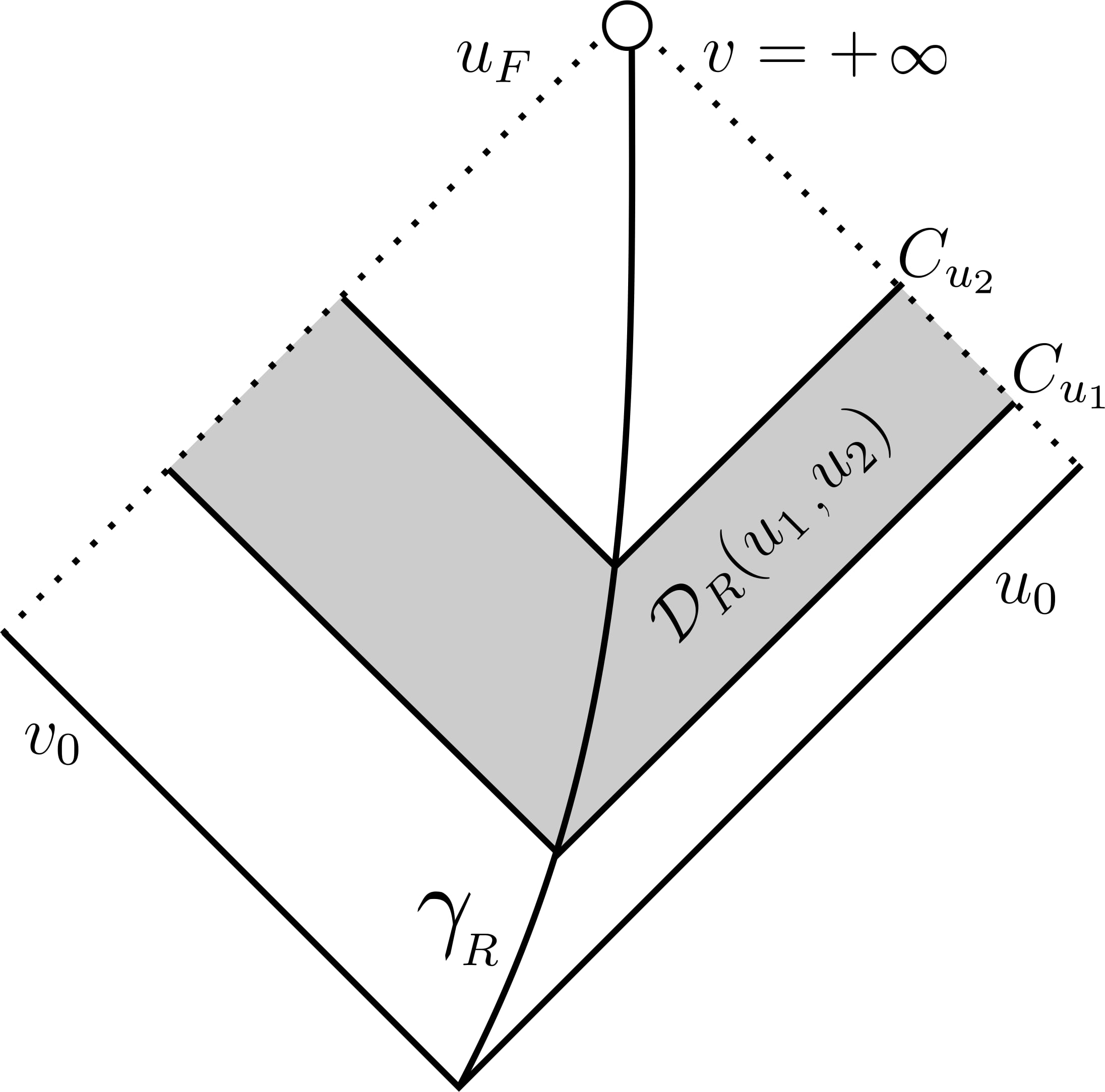}	
			\end{center}
		\end{figure}		
	\subsection{The principles of the  $r^p$ method}\label{section.rp} On Minkowski spacetime, the wave equation $-\rd_{t}^2 \phi + \rd_{x}^2 \phi +\rd_{y}^2 \phi +\rd_{z}^2 \phi=0$ takes the following form, where $u=t-r$, $v=t+r$, $r=\sqrt{x^2+y^2+z^2}$, $\psi= r\phi$ (radiation field) and $\sDelta$ is the standard Laplacian on $\mathbb{S}^2$: \begin{equation}\label{wave.mink}
	\rd_u \rd_v \psi =r^{-2}\sDelta \psi.
	\end{equation} The key idea of the $r^p$ method of Dafermos and Rodnianski \cite{rp} is to prove decay-in-time for $\phi$  using energy methods, namely  $L^2$-based estimates, which can  typically be carried out on more general spacetimes, with a potential or a nonlinearity. However, the standard energy $E(T)$  computed on the time-slice $\{t=T\}$ is constant and does not decay. Instead, one can compute the energy on a $V$-shaped foliation (indexed by $u$) as depicted in Figure~\ref{fig} and the goal is to prove that this energy $\tilde{E}(u)$ ends up decaying in $u$. To capture this decay, Dafermos and Rodnianski use the multiplier $r^p \rd_v \psi$ and, upon integration on a spacetime region $\{u_1 \leq u \leq u_2\}$, obtain \begin{equation}\label{rp.intro}
 \int_{u=u_2} r^p|\rd_v \psi|^2+  \int_{u_1}^{u_2} [\int_{u=u'}   r^{p-1} |\rd_v \psi|^2] du'\lesssim     \int_{u=u_1}  r^p|\rd_v \psi|^2+ \tilde{E}(u_1),
	\end{equation} for all $u_1<u_2$, where $0 < p\leq 2$, and \eqref{ILED} is used to absorb various boundary terms at the curve $\{r=R\}$ into the energy $\tilde{E}(u_1)$. \eqref{rp.intro} consists of two statements: the boundedness of the $r^p$-weighted energy for $0\leq p\leq 2$, and the (integrated) decay of the $r^{p-1}$ weighted energy. Using \eqref{rp.intro} as a hierarchy of estimates in $p$ (see \cite{rp} for details), Dafermos and Rodnianski obtain the time-decay of the energy at the rate $u^{-p}$, where here $p=2$: \begin{equation}\label{decay.intro}
	\tilde{E}(u) \lesssim u^{-2}.
	\end{equation} We emphasize that the proof of \eqref{decay.intro} relies heavily on the validity of the energy boundedness \eqref{boundedness} and integrated local decay \eqref{ILED} estimates, used as a black box. The argument of the $r^p$ method can be also be carried out on the Schwarzschild or  (sub-extremal) Reissner--Nordstr\"{o}m spacetime \eqref{RN} and in higher dimensions, see \cite{Volkerrp} and 	\cite{Moschidisrp} for the more general setting in which it can be applied.
	
In \cite{AAG0,AAG1}, Angelopoulos, Aretakis and Gajic introduced a novel viewpoint on the $r^p$ method allowing  to prove sharp decay and late-time tails for a large class of wave equations \begin{equation}\label{subcrit}
	\Box_g \phi = V(r) \phi, \text{  where } |V|(r) \lesssim r^{-2-\ep} \text{ as } r \rightarrow +\infty,
\end{equation} on a stationary and spherically-symmetric asymptotically flat  metric $g$  and a scale sub-critical potential $|V|(r) \lesssim r^{-2-\ep}$, where $\ep>0$. 
 See also \cite{AAG2,Ma3,Ma4} for follow-up work and extensions of these methods. We also mention the recent \cite{twotails} establishing late-time tails for a large class of equations which is even more general than \eqref{subcrit} (for instance also allowing for first-order terms and nonlinearities) using a novel approach.
 
 \subsection{Our new strategy to address scale-critical potentials} Following the spirit of \cite{rp}, we make the assumptions of energy boundedness \eqref{boundedness} and integrated local decay \eqref{ILED} that we use as a black box. Using the multiplier $r^p\rd_v$ for \eqref{wave.main} gives schematically (compare with \eqref{rp.intro})  \begin{equation}\label{rp.intro2}
 	\int_{u=u_2} r^p|\rd_v \psi|^2+  \int_{u_1}^{u_2} [\int_{u=u'}   r^{p-1} |\rd_v \psi|^2] du'\lesssim     \int_{u=u_1}  r^p|\rd_v \psi|^2+ \tilde{E}(u_1) + \ep \int_{u_1}^{u_2} [\int_{u=u'}   r^{p-2} |\psi| |\rd_v \psi|] du',
 \end{equation}
 where we have taken $w_1=W_1=q=Q=0$ for now to simplify and $\ep \int_{u_1}^{u_2} [\int_{u=u'}   r^{p-2} |\psi| |\rd_v \psi|] du$ is an integrated (so-called ``bulk'') error term.
 The standard approach of the $r^p$ method (see \cite{AAG0}) is to absorb the bulk error term  into the $\int_{u_1}^{u_2} [\int_{u=u'}   r^{p-1} |\rd_v \psi|^2]$ term in the left-hand-side, which works well for  a sub-critical potential satisfying \eqref{subcrit}, with the use of the Hardy inequality (see Proposition~\ref{Hardy.prop}). Here, the Hardy inequality gives  \begin{equation}\label{hardy.intro}\begin{split}
 	& \int_{u_1}^{u_2} [\int_{u=u'}   r^{p-2} |\psi| |\rd_v \psi|] du' \lesssim\int_{u_1}^{u_2} [\int_{u=u'}   r^{p-1}  |\rd_v \psi|^2] du'+  \int_{u_1}^{u_2} [\int_{u=u'}   r^{p-3} |\psi|^2 ] du'\\ & \lesssim  \int_{u_1}^{u_2} [\int_{u=u'}   r^{p-1}  |\rd_v \psi|^2] du' + \tilde{E}(u_1),\end{split}
 \end{equation} where $0\leq p<2$. To absorb the error term $\ep \int_{u_1}^{u_2} [\int_{u=u'}   r^{p-1} |\psi| |\rd_v \psi|] du'$ into the left-hand-side with the help of \eqref{hardy.intro}, we must further restrict  $p<2-O(\sqrt{\ep})$, which gives the $r^p$-hierarchy in a range $0<p< 2-O(\sqrt{\ep})$ \begin{equation}\label{rp.intro3}
 \int_{u=u_2} r^p|\rd_v \psi|^2+  \int_{u_1}^{u_2} [\int_{u=u'}   r^{p-1} |\rd_v \psi|^2] du'\lesssim     \int_{u=u_1}  r^p|\rd_v \psi|^2+ \tilde{E}(u_1),
 \end{equation} which is sufficient to obtain \eqref{decay.intro}, up to an arbitrarily small loss $O(u^{O(\sqrt{\ep})})$. \eqref{decay.intro} is one power of $u$ away from sharpness, however (compare with \eqref{energy.est}). To go up to $p<3$ and obtain sharper estimates, we must decompose the solution $\phi$ into its spherical average and higher angular modes, following the original idea of \cite{AAG0}: \begin{equation}
 \phi(u,v,\omega)= \underbrace{\phi_0(u,v)}_{:= \int_{\mathbb{S}^2} \phi(u,v,\omega) d\omega}+ \phi_{\geq 1}(u,v,\omega).
 \end{equation} It turns out that $r^2 \rd_v(r\phi_{\geq 1})$ also obeys the hierarchy \eqref{rp.intro3}, translating\footnote{We note that this step requires to commute \eqref{wave.main} with $\rd_v$, which is why \eqref{H1} requires assumptions on the $\rd_v$ derivatives of the potential terms. These assumptions, however, are not necessary for spherically-symmetric solutions of \eqref{wave.main}.} into faster energy decay for  $\phi_{\geq 1}$, i.e.,
 \begin{equation}\label{faster.intro}
\tilde{E}[\phi_{\geq 1}](u) \lesssim u^{-4+ O(\sqrt{\ep})},
 \end{equation} and \eqref{faster.intro} is essentially sufficient to show \eqref{energy.est}, \eqref{pointwise.est}, \eqref{pointwise.est2} for $\phi_{\geq 1}$. Therefore, our next objective is to prove \eqref{rp.intro3} for the spherical average $\psi_0$ and $2<p<3$. The novelty of our approach is to absorb the bulk error term into  	$\int_{u=u_2} r^p|\rd_v \psi|^2$ instead of $\int_{u_1}^{u_2} [\int_{u=u'}   r^{p-1} |\rd_v \psi|^2] du'$ using a  Grönwall argument. Starting from \eqref{rp.intro2}, we proceed,  using the Hardy inequality  differently from \eqref{hardy.intro} to get schematically, exploiting \eqref{rp.intro3} with $p-1$
 \begin{equation}\label{hardy.intro2}\begin{split}
 &	\int_{u_1}^{u_2} [\int_{u=u'}   r^{p-2} |\psi| |\rd_v \psi|] du' \lesssim[\int_{u_1}^{u_2} [\int_{u=u'}   r^{p}  |\rd_v \psi|^2] du']^{\frac{1}{2}} [\int_{u_1}^{u_2} [\int_{u=u'}   r^{p-4}  | \psi|^2] du']^{\frac{1}{2}}\\ &\ls [\int_{u_1}^{u_2} [\int_{u=u'}   r^{p}  |\rd_v \psi|^2] du']^{\frac{1}{2}} [\int_{u_1}^{u_2} [\int_{u=u'}   r^{p-2}  | \rd_v\psi|^2] du']^{\frac{1}{2}} \lesssim [\int_{u_1}^{u_2} [\int_{u=u'}   r^{p}  |\rd_v \psi|^2] du']^{\frac{1}{2}} [\int_{u=u_1}   r^{p-1}  |\rd_v \psi|^2+ \tilde{E}(u_1)]^{\frac{1}{2}} .\end{split}
 \end{equation} To conclude the proof of \eqref{energy.est}, we close the $r^p$-weighted hierarchy \eqref{rp.intro3} for $0<p< 3-O(\sqrt{\ep})$ using an induction argument, where $u_1$ and $u_2$ take values on a dyadic sequence. The proof of \eqref{pointwise.est} follows from \eqref{energy.est} by a standard argument. To obtain the sharp point-wise decay on a constant-$r$ curve \eqref{pointwise.est2}, we must retrieve faster decay for the time-derivative $T\phi_0$, i.e, \begin{equation}\label{energy.est.T0}	E[T\phi_0](u)  \lesssim u^{-5+\eta(\ep)},\end{equation} which is done considering a $r^p$-weighted hierarchy for $r\rd_v(r\phi_0)$ (see \cite{inverse.Dejan} where this hierarchy was introduced).
 
 We finally note that we additionally allow for first-order linear terms in \eqref{wave.main}. The ingoing derivative terms in \eqref{wave.main} are handled with the use of a novel ``Hardy inequality'' (\eqref{Hardy.in} in Proposition~\ref{Hardy.prop}) which only works for solutions to the wave equation \eqref{wave.main}.
	
	\subsection{Examples of spacetimes and potentials satisfying the assumptions}\label{section.example}
	
		 \subsubsection{Decay rates in the assumptions}
	For the sake of comparison with the Minkowski metric, let us define a time-variable $t=v+u$, and we keep in mind that $r$ is comparable to $v-u$ (for large $r$). We are interested in a region where $u\geq u_0$, $v-u\geq R$, thus \begin{equation}
		r(u,v)\lesssim u\lesssim t(u,v) \lesssim v.
	\end{equation} Thus, the $r$-decay assumptions in Theorem~\ref{main.thm} represent the weakest possible form of decay. In the presence of nonlinearities, where Theorem~\ref{main.thm} can be adapted, the decay of some coefficients may be obtained in terms of inverse powers of $t$, which is stronger than the inverse powers of $r$ required in \eqref{H0}, \eqref{H1}, and \eqref{H3}.
	
	Finally, note that the scaling in \eqref{wave.main} is arranged to respect the radiative structure of the wave equation which gives a finite limit for $r\phi$ and $r^2 \rd_v (r\phi)$ towards $v=+\infty$  so that, for fixed $u$, the three terms $\phi$, $\rd_u \phi$, $r\rd_v \phi$ share the same scaling, i.e., \begin{equation}
		|\phi|(u,v,\omega) \lesssim r^{-1},\ |\rd_u\phi|(u,v,\omega) \lesssim r^{-1},\ r|\rd_v\phi|(u,v,\omega) \lesssim r^{-1} \text{ as } v \rightarrow+\infty.
	\end{equation}
	
	\subsubsection{Metric assumptions}

		We first note that there exist many examples of Lorentzian metrics $g$ satisfying \eqref{H0}, such as the (sub-extremal) Reissner--Nordstr\"{o}m metric, modeling the exterior of a charged star/black hole for $r\geq r_+(M,e)$ \begin{equation}\label{RN}
			g_{RN} =- (1-\frac{2M}{r}+ \frac{e^2}{r^2}) dt^2+ (1-\frac{2M}{r}+ \frac{e^2}{r^2})^{-1} dr^2 + r^2 d\sigma_{\mathbb{S}^2},
		\end{equation} where $0\leq |e|\leq M$ and $r_+(M,e):= M+\sqrt{M^2-e^2}$. Indeed, introducing the standard coordinates \begin{equation}
		u=t-r^{*},\ v= t+ r^{*},\ \frac{dr^{*}}{dr}= (1-\frac{2M}{r}+ \frac{e^2}{r^2})^{-1},
		\end{equation} shows that \eqref{RN}assumes the form of $g$ given in \eqref{wave.main} with \begin{equation}
		\Omega^2(u,v) = 4  (1-\frac{2M}{r}+ \frac{e^2}{r^2}),\ \rd_v r= -\rd_u r=   (1-\frac{2M}{r}+ \frac{e^2}{r^2}).
		\end{equation}

		 Note that in the case $e=0$, \eqref{RN} reduces to the well-known Schwarzschild metric, which itself  reduces to the (trivial) Minkowski metric $g= -dt^2 + dx^2+dy^2+dz^2$ for $M=0$,  under which $\Box_g = -\rd_{t}^2+\rd_{x}^2+\rd_{y}^2+\rd_{z}^2$.

		 Lastly, we note that a large class of non-stationary spacetimes will also satisfy \eqref{H0}, in particular spacetimes which converge to, or remain close to the Reissner--Nordstr\"{o}m metric \eqref{RN} at large-time. 

		 \subsubsection{Integrated local decay and energy boundedness}
		 
	It is well-known \cite{claylecturenotes} that	the energy boundedness \eqref{boundedness} and integrated local energy estimate \eqref{ILED} hold on the Schwarzschild/(sub-extremal) Reissner--Nordstr\"{o}m metric  in the absence of a potential, i.e., for $\Box_{g_{RN}}\phi=0$.
	
	In \cite{Moi2},   \eqref{boundedness} and  \eqref{ILED} are established\footnote{While the proof of \cite{Moi2} is, strictly speaking, only for the charged scalar field equation, it is easy to generalize it to \eqref{wave.main}.} for spherically-symmetric solutions of \eqref{wave.main} on a (sub-extremal) Reissner--Nordstr\"{o}m metric, i.e., for $g=g_{RN}$ and with $w_1=W_1=0$, under the smallness condition of $\ep$.

	In \cite{inverse.Dejan}, Gajic established \eqref{boundedness} and  \eqref{ILED}  for \eqref{wave.main} on the sub-extremal Reissner--Nordstr\"{o}m metric  with $w_1=W_1=q=Q=0$, and $\ep=1$ (no smallness condition), under the assumption that $V(r)= r^{-2} (w_0(r) + W_0(r))$ is time-independent and  $w_0(r) + W_0(r)$ admits a limit as $r\rightarrow \infty$, i.e., there exists $\alpha>-\frac{1}{4} $ such that \begin{equation}\label{asymp.V}
V(r) \sim \alpha r^{-2} \text{ as } r \rightarrow+\infty,
	\end{equation}\begin{equation}\label{global.V}
\text{ and }	V(r) \geq -\frac{1}{4r^2} \text{ for all } r \geq r_+(M,e) = M +\sqrt{M^2-e^2}.
	\end{equation}
		 It is also possible to replace \eqref{global.V} with a no-resonance condition on $V(r)$, see \cite{inverse.Dejan}.

		 Coming back to more general considerations, we discuss the $T$-commuted energy boundedness \eqref{boundedness.T} and integrated local decay \eqref{ILED.T}. Note that, if the metric and  the potentials are stationary, i.e., \begin{equation}\label{stationary}
		 	Tr=0,\ T\Omega^2=0,\ Tw_i=Tq=TW_i=Tq=0,
		 \end{equation} then  \eqref{boundedness.T}, \eqref{ILED.T} immediately follow from  \eqref{boundedness} and \eqref{ILED}, in fact, we  have the stronger estimate: \begin{equation}
\begin{split}
	&  \int_{\mathcal{D}_{u_1,\infty}}  \frac{|T\phi|^2(u,v,\omega)}{(1+r)^{\sigma'}} du dv d\omega  \leq C_{\sigma'} E[T\phi](u_1),\\ & E[\phi](u_2) \leq D E[\phi](u_1).
\end{split}
		 \end{equation} More generally, it is possible to retrieve \eqref{boundedness.T}, \eqref{ILED.T} if the quantities involved in \eqref{stationary} decay in time at a rate $O(u^{-1+O(\sqrt{\ep})})$. Such a result should, however, be obtained on a case-by-case basis.
	
		 \subsubsection{Potential assumptions for sharp energy and radiation field decay}
	We now discuss \eqref{H1}, the potential assumptions used	 to obtain the first conclusions of Theorem~\ref{main.thm}, i.e., \eqref{energy.est}, \eqref{pointwise.est}. We note that \eqref{H1} allows for a large class of time-dependent potentials.	  The most general form for $w_0$ to respect \eqref{H1} allows for linear oscillations in $u$ and logarithmic oscillations in $v$ or $r$. $q$ and $w_1$ are also allowed to oscillate linearly in $u$, however, they cannot oscillate in $r$. In other words,   \begin{equation}\label{h1.intro}
		 	w_0(u,v) = f_0\big( u, \log(r(u,v))\big),\ w_1(u,v) = f_1( u,r^{-1}(u,v)),\  q(u,v) = f_q( u,r^{-1}(u,v)),
		 \end{equation} where $f_0(X,Y)$, $f_1(X,Y)$, $f_q(X,Y)$ are bounded functions and $\rd_Y f_0$, $\rd_Y f_q$, $\rd_Y f_1$  are bounded. For $W_0$, $W_1$ and $Q$, we can take, for any $\ep>0$ and $F_0(X,Y)$, $F_1(X,Y)$, $F_Q(X,Y)$  bounded with $\rd_Y F_0$, $\rd_Y F_1$, $\rd_Y F_q$  bounded \begin{equation}\label{h1.intro2}
		 W_0(u,v) = (1+r)^{-\ep}F_0\big( u, \log(r(u,v))\big),\ W_1(u,v) =F_1( u,r^{-1-\ep}(u,v)),\  Q(u,v) =F_Q( u,r^{-1-\ep}(u,v)).
		 \end{equation}
		 
		 		 \subsubsection{Potential assumptions for sharp point-wise scalar field decay}

 To obtain the stronger conclusion of Theorem~\ref{main.thm}, i.e., \eqref{pointwise.est2}, we require \eqref{H3}, which we now discuss. The key difference, compared to the less demanding assumptions \eqref{H1}, is that the potentials are no longer allowed to feature linear oscillations in $u$, only logarithmic ones, i.e., \begin{equation}
 	w_0(u,v) = g_0\big( \log(u), \log(r(u,v))\big),\ w_1(u,v) = g_1( \log(u),r^{-1}(u,v)),\  q(u,v) = g_q( u,r^{-1}(u,v)),
 \end{equation} where $g_0(X,Y)$, $g_1(X,Y)$, $g_Q(X,Y)$ are bounded functions with bounded derivative (compared with \eqref{h1.intro}).  $W_0$, $W_1$ and $Q$ obey similarly stronger assumptions compared to \eqref{h1.intro2}, which we omit to state.
	
	\subsection{Previous works on scale-critical potentials and sharpness of the decay}\label{section.sharp}
	
	The literature on decay estimates for scale-critical potentials is vast, we refer to the review \cite{wave.Schlag}.

For the wave equation \eqref{wave.main} with an exact inverse-square potential  on Minkowski spacetime, i.e., \begin{equation}
		g= -dt^2 + dx^2+ dy^2+dz^2,\  w_0(u,v)=1,\ w_1=W_0=W_1=q=Q=0,
	\end{equation} it is known \cite{Dean,MoiDejan} that the following asymptotics  hold \begin{equation}\label{sharp}
		r|\phi|(u,v,\omega) \approx \left( \frac{r}{u v}\right)^{\frac{1+\sqrt{1+4\ep}}{2} }.
	\end{equation} 
	
		Moreover, \eqref{sharp} also holds on the sub-extremal Reissner--Nordstr\"{o}m metric \cite{inverse.Dejan} for a large class of time-independent asymptotically inverse-square potentials satisfying \eqref{asymp.V}, \eqref{global.V}, see \cite{Schlag.scalecrit} for previous works involving sharp upper bounds. Lastly, precise late-time tails for an even more general class of stationary spacetimes and time-independent scale-critical potentials were proved by Hintz \cite{Hintz.new}.

	In comparison, Theorem~\ref{main.thm} allows for a class of time-dependent potentials that are allowed to oscillate,  as discussed earlier in this section; however, we do not derive precise late-time tails. Our approach moreover relies on the smallness of $\ep$, and the estimates \eqref{energy.est}, \eqref{pointwise.est}, \eqref{pointwise.est2} are not, strictly speaking, sharp in terms of decay rate, because they feature an arbitrarily small loss $O(u^{\eta(\ep)})$, where $\eta(\ep) = o(1)$ as $\epsilon \rightarrow 0$. However, the above examples satisfying the estimates \eqref{sharp} show that our  estimates \eqref{energy.est},  \eqref{pointwise.est}, \eqref{pointwise.est2} in Theorem~\ref{main.thm} are sharp, up to this arbitrarily small loss.  Note, in particular, that on a constant-$r$ curve, \eqref{sharp} gives \begin{equation}
		|\phi|(u,v,\omega)\lesssim u^{-2+ O(\ep)}.
	\end{equation}
	Finally, we notice that the Maxwell-charged-scalar-field equations are modeled after a wave equation with a scale-critical of the form \eqref{wave.main}, where $\ep$ represents the asymptotic charge of the spacetime. This system was studied by the author \cite{Moi2} in spherical symmetry on a Reissner--Nordstr\"{o}m black hole assuming $\ep$ is small. On the other hand, for  the Maxwell-charged-scalar-field equations  on Minkowski spacetime, $\ep$ is schematically a time-dependent function $\ep=Q(t) \rightarrow 0$ as $t\rightarrow+\infty$; in this context, Yang and Yu proved global existence  with no smallness assumption on the initial data \cite{YangYu}; see also \cite{LindbladKG,Bieri.KG,ShiwuKG} for previous works.

	\subsection{Acknowledgments} The author gratefully acknowledges  support from the NSF Grant DMS-2247376. Special thanks go to Hayd\'{e}e Pacheco for the figure.

		\section{Preliminary}\label{prelim.section}
		
		We start by re-writing \eqref{wave.main} in terms of the radiation field $\psi:= r\phi$: we find the formula
		
		\begin{equation}\label{psi.eq}
\rd_u \rd_v \psi = \frac{\Omega^2}{4}r^{-2}\slashed{\Delta}_{\mathbb{S}^2} \psi     + \frac{\ep}{r^2}\left(\sum_{i=0}^{1}
  \tw_i(u,v) \cdot   \rd_u^{i}\psi+ 
    \tq(u,v) \cdot r \rd_v \psi \right),
		\end{equation}  where we defined 
		\begin{equation*}
			\begin{split}
			&	\tw_i(u,v) = \cw_i(u,v) + \ep^{-1} \cW_i(u,v)\\& 
				\tq(u,v) = \cq(u,v) + \ep^{-1} \cQ(u,v)
			\end{split}
		\end{equation*} and
		
		\begin{equation*}\begin{split}& \cw_0(u,v) = -\frac{\Omega^2}{4} w_0(u,v)+\frac{[\rd_u r ]\Omega^2}{4}  w_1(u,v)+ \frac{[\rd_v r]\Omega^2}{4} q(u,v),\ \tw_1(u,v) = -\frac{\Omega^2}{4} w_1(u,v),\\ & \cW_0(u,v) = -\frac{\Omega^2}{4} W_0(u,v)+ \frac{[\rd_u r ]\Omega^2}{4}  W_1(u,v)+  \frac{[\rd_v r]\Omega^2}{4}+ r \rd_u \rd_{v} r,\ \cW_1(u,v) = -\frac{\Omega^2}{4} W_1(u,v),\\ &  \cq(u,v) = -\frac{\Omega^2}{4} q(u,v),\ \cQ(u,v) = -\frac{\Omega^2}{4} Q(u,v).		\end{split}		\end{equation*}
		
		 Note that, under the assumptions \eqref{H0}, \eqref{H1}, we have 
		 	\begin{equation}\label{H2}
		 		  |\tw_i|(u,v),\ |\tq|(u,v),\ r|\rd_v \tw_0|(u,v),\ r|\rd_v \tq|(u,v),\ r^2|\rd_v \tw_1|(u,v)  \lesssim 1,
		 \end{equation} and under the assumption \eqref{H3}, \begin{equation}\label{H4}
		 |\rd_u \tw_i|(u,v),\ |\rd_u \tq|(u,v),\   \lesssim r^{-1}(u,v),
		 \end{equation}	
	in the region where $ r\geq R(\epsilon)$, where we assume $R(\ep)>0$ to be a large constant. We then  define the null outgoing cone  \begin{equation}
			C_u = \{ u'=u,\ v \geq v_R(u),\ \omega \in \mathbb{S}^2\},
		\end{equation} where $v_R(u)=u+R$ is such that $\rho(u,v_R(u))=R$, where $\rho=v-u$. We will use the notation $\gamma_R=\{(u,v_R(u),\omega),\ u \geq u_0,\ \omega \in \mathbb{S}^2\}=\{(u_R(v),v,\omega),\ v \geq v_0,\ \omega \in \mathbb{S}^2\}$.
		We also define  global  $V$-shaped foliation		   $\N_u$  \begin{equation}\label{N.def}
			\N_u  = C_u \cup  \underbrace{\{ u' \leq u,\ v'=v_R(u) ,\ \omega \in \mathbb{S}^2\}}_{\underline{C}_u},
		\end{equation}
		 where $u_R(v)$ is such that $r(u_R(v),v)=R$;
	and for any $u_1<u_2$, the spacetime domain \begin{equation}\label{D.def}
\DD(u_1,u_2)=\{ u_1\leq u\leq u_2,\ v\geq v_R(u),\ \omega \in \mathbb{S}^2\}.
	\end{equation}
	and the unweighted energy on $\N_u$ defined as  \begin{equation}
		E[\phi](u) =  \int_{C_u} ( r^2|\rd_v \phi|^2+ |\snabla \phi|^2)  dv d\omega+ \int_{\underline{C}_u} ( r^2|\rd_u\phi|^2+ |\snabla \phi|^2)  du' d\omega
	\end{equation}
	 \begin{equation}
		E_p[\psi](u) = \int_{C_u} r^{p} |\rd_v \psi|^2 dv d\omega,
	\end{equation}
	\begin{equation}
		\tilde{E}_p[\phi](u) = 	E_p[r\phi](u) + E(u).
	\end{equation}
	
	Our convention is that the volume form that we use is always $du dv d\sigma_{\mathbb{S}^2}$.

	For any function $f \in L^2(\mathbb{S}^2)$, we write its decomposition in terms of spherical harmonics $\omega \in \mathbb{S}^2 \rightarrow Y_L(\omega)$ \begin{equation}\begin{split}
&f(\omega) = \sum_{L=0}^{+\infty} f_L Y_L(\omega),\\ & f_{\leq L_0 }(\omega) = \sum_{L=0}^{L_0} f_L Y_L(\omega),\\ & f_{\geq L_0 }(\omega)=\sum^{+\infty}_{L=L_0} f_L Y_L(\omega), \end{split}
	\end{equation} for any $L_0 \in \mathbb{N}$, and, of course, the notation extends to $f(u,v,\omega)$ for which $f_L(u,v)$ now also depends on $(u,v)$.
	
	We will also always adopt the convention $A\ls B$ if there exists a constant $C>0$ that it independent of $\ep$ such that \begin{equation}
		A \leq C \cdot B.
	\end{equation} We will also write $A\pm O(f(\ep)$ as a replacement for $A \pm f(\ep)$, where $f(x)\geq 0$ and $f(0)=0$.
	
	Finally, with no loss of generality, we will always assume that $u_0>1$.
	\section{Boundedness of the $r^p$-weighted energy and energy decay}

	\subsection{$r^p$-weighted energy and  black box  decay results}
		In this section, we provide preliminary calculus results that will be useful in the sequel. The key idea is to succeed in converting the boundedness of the $r^p$-weighted energy into the time-decay of the unweighted energy.
	
	\begin{lemma}[\cite{Moi2}, Lemma 6.3] \label{lemma.hierarchy}
		Suppose that there exists $1<p<2$  such that for all $u_0\leq u_1<u_2$
		
		\begin{equation} \label{lemmadecay}
			\int_{u_1}^{u_2} E_{p-1}[\psi](u)du + \tilde{E}_{p}(u_2)  \lesssim \tilde{E}_{p}(u_1). 
		\end{equation}

		Then for all $0 \leq q \leq p$ and for all $k \in \mathbb{N}$ 
		
		\begin{equation}
			\tilde{E}_{q}(u) \lesssim  u^{-k(1-\frac{q}{p})}+ \frac{\sup_{ 2^{-2k-1} u \leq u' \leq u} \tilde{E}_{p}(u')}{u^{p-q}} \lesssim u^{-p+q}\cdot \tilde{E}_{p}(u_0).
		\end{equation}
	\end{lemma}
	Now, we move to a point-wise decay result taking advantage of the decay of the weighted energy.
	\begin{prop}\label{energy.decay.prop}
		For any $\gamma \in (0,\frac{1}{2})$, and $v\geq v_R(u)$, the following estimate holds: \begin{equation}\label{pointwise1}
			r^{\frac{1}{2} + \gamma}(u,v) \|\phi(u,v,\cdot)\|_{L^2(\mathbb{S}^2)}\lesssim \left[\tilde{E}_{2\gamma}(u)\right]^{\frac{1}{2}}.
		\end{equation}
	\end{prop}
	
	\begin{proof}
		This is a standard integration argument, see for instance \cite{Moi2}, Lemma 7.1.
	\end{proof}
	
	Then, we include a standard consequence of the Integrated Local Decay Estimate \eqref{ILED} to control the energy on a timelike curve, using an averaging argument in $r$, see e.g.\ \cite{Moi2}, Proposition 6.4.
	
	\begin{lemma}\label{lemma.average} The following estimates hold true for all $u_1<u_2$, and $k\in \mathbb{N} \cup \{0\}$:
\begin{equation}\label{app.ILED}
\sum_{|\beta|\leq k}	\int_{u_1}^{u_2} \sum_{|\alpha|\leq 1} \int_{\mathbb{S}^2}|\rd_{\alpha}\snabla^{\beta} \phi|^2(u,v_{R}(u),\omega) du d\omega \lesssim \sum_{|\beta|\leq k}  E[\snabla^{\beta} \phi](u_1).
\end{equation}
	\end{lemma}\begin{proof}
While the argument is already contained in the proof of   Proposition 6.4 in  \cite{Moi2},  we briefly sketch it for the reader's convenience. We will always assume $R$ is large enough.	By \eqref{ILED}, we have, recalling $\rho(u,v)=v-u$, $$ \int_{ \frac{R}{2}\leq \rho \leq 2R,\ u_1 \leq u \leq u_2} \sum_{|\alpha|\leq 1} |\rd_{\alpha} \phi|^2 \leq C_R E(u_1).$$ So by the mean-value theorem applied to the function $\rho$, there exits $R^{*} \in ( \frac{R}{2}, 2R)$ such that
$$\int_{u_1}^{u_2} \sum_{|\alpha|\leq 1} |\rd_{\alpha} \phi|^2(u,v_{R^{*}}(u)) du \leq  2 R C_R E(u_1).$$ 
With no loss of generality, we can then replace $R$ by $R^{*}$ and thus \eqref{app.ILED} is proved for $k=0$. Similarly, after commuting \eqref{wave.main} with $\snabla^{\beta}$, $|\beta|\leq k$  and applying \eqref{ILED} to this new solution, we obtain \eqref{app.ILED} for any $k \in \mathbb{N}$.
	\end{proof}

	\subsection{$r^p$-weighted multipliers}
	In this section, we assume sufficient regularity of all the functions involved so that integration by part makes sense. Our goal is to derive various integrated identities using $r^p \rd_v$ multipliers for \eqref{psi.eq} or commuted versions.
	\begin{lemma}\label{lemma.calc1}
		Assume that \begin{equation}\label{psi.source}
			\rd_u \rd_v \psi =\frac{\Omega^2}{4} r^{-2} \sDelta \psi  + F
		\end{equation}
		
		Then \begin{equation}\label{rp1}\begin{split}
&\int_{\mathbb{S}^2}  \left[	\rd_u (\frac{ r^p|\rd_v \psi|^2}{2})+ 	\rd_v (\frac{\Omega^2 r^{p-2}|\snabla \psi|^2}{8})  + \Omega^2(	\frac{p}{2} r^{p-1} |\rd_v \psi|^2+	\frac{1}{8}\big([2-p] \rd_v r -  r\rd_v \log(\Omega^2) \big)r^{p-3}|\snabla \psi|^2)\right]\\ & =  \int_{\mathbb{S}^2}  r^p \rd_v \psi\ F.\end{split}
		\end{equation}
		
	\end{lemma}\begin{proof}
		Multiply \eqref{psi.source} with $r^p \rd_v \psi$ and integrate on $\mathbb{S}^2$.
	\end{proof}
	
	\begin{lemma}\label{lemma.calc2}
		Assume that $\psi$ satisfies \eqref{psi.source}. Then, defining  $\Psi_1 := \Omega^{-2}r^2 \rd_v \psi$, we have
		\begin{equation}\label{psi.source.commuted}
			\rd_u  \Psi_1 + 2r^{-1} [-\rd_u r] \Psi_1 = \frac{1}{4}\sDelta \psi + \frac{r^2 F}{\Omega^2}
		\end{equation}
			\begin{equation}\label{psi.source.commuted2}
			\rd_u \rd_v  \Psi_1 + 2r^{-1} [-\rd_u r] \rd_v \Psi_1- 2r^{-2} [-\rd_v r \rd_u r  + r \rd_{uv}^2 r] \Psi_1= \frac{\Omega^2}{4}r^{-2}\sDelta \Psi_1 + \rd_v(\frac{r^2 F}{\Omega^2})
		\end{equation}
		Therefore, the use of the $r^p\rd_v$ multiplier provides the following identity:\begin{equation}\label{rp2}\begin{split}
	&		\int_{\mathbb{S}^2} \big(	(-\rd_u r)[2+\frac{p}{2}] r^{p-1} |\rd_v \Psi_1|^2+		\rd_u (\frac{ r^p|\rd_v \Psi_1|^2}{2}) +\frac{\Omega^2}{8}\big([2-p] \rd_v r -  r\rd_v \log(\Omega^2) \big)r^{p-3}|\snabla \Psi_1|^2)\\ &+  r^{p-3}|\Psi_1|^2\left[(p-2)\rd_v r [-\rd_v r \rd_u r  + r \rd_u \rd_v r] - (\rd_{vv}^2 r)(\rd_u r)+ r ( \rd_u \rd_v \rd_v  r)\right]\big) \\ &=     \int_{\mathbb{S}^2}   r^p \rd_v \Psi_1\ \rd_v(\frac{r^2F}{\Omega^2})+\int_{\mathbb{S}^2}  \rd_v\left( r^{p-2}\big( -\rd_u r \rd_v r+ r \rd_u \rd_v r\big)|\Psi_1|^2- \frac{\Omega^2 r^{p-2}}{8} |\snabla \Psi_1|^2\right)  \end{split}
		\end{equation}
	\end{lemma}\begin{proof}	The proof is similar to that of Lemma~\ref{lemma.calc1}.
	\end{proof}

	\subsection{Hardy and Poincar\'{e} inequalities}
	
	In this section, we recall the standard Hardy and Poincar\'{e} inequalities (that do not require \eqref{wave.main} to be satisfied), as expressed in the language we will be using. We also include a less trivial and novel result, involving a Hardy-type inequality for the ingoing derivative taking advantage of \eqref{wave.main}.
	
	\begin{prop} \label{Hardy.prop}Let $q \neq 2$ and assume that for all $u\geq u_0$, $\omega \in \mathbb{S}^2$ $$ \lim_{v\rightarrow+\infty} r^{q-2}(u,v)f(u,v,\omega)=0.$$ Then
	\begin{equation} \label{Hardy5}\begin{split}
		 \int_{\DD(u_1,u_2)}r^{q-3}  | f|^2 du dv \lesssim            |2-q|^{-2} \int_{\DD(u_1,u_2)}r^{q-1}  |\rd_v f|^2 du dv    +                       \frac{R^{q-2}}{|2-q|} \int_{u_1}^{u_2}|f|^2 (u,v_R(u)) du .  \end{split}
\end{equation}

Let $q<2$ and $\phi$ a solution of \eqref{wave.main}, with $\psi=r\phi$. Then
\begin{equation}\label{Hardy.in}\begin{split}
		&
		\int_{\DD(u_1,u_2)} r^{q-3} |\rd_u \psi|^2  \lesssim \ep^2 (2-q)^{-1}\int_{\DD(u_1,u_2)}\left[  r^{q-3} |\rd_v  \psi|^2+ r^{q-5} |\snabla \psi|^2\right]  \\  & +(2-q)^{-1}\left(\int_{C_{u_1}} r^{q-4}  |\snabla \psi|^2+R^{p-4}\int_{u_1}^{u_2} (R^2|\rd_u \psi|^2+| \psi|^2+ |\snabla \psi|^2)(u,v_R(u)) du\right)
	\end{split}
\end{equation}
	\end{prop}

	\begin{proof}The proof of  the Hardy inequality \eqref{Hardy5}  is standard, see for instance \cite{Moi2}, Lemma 2.2. For \eqref{Hardy.in}, we use a similar strategy and write:
	
		\begin{equation*}\begin{split}
			&\int_{\DD(u_1,u_2)} r^{p-3} |\rd_u \psi|^2=-\frac{1}{2-p} \int_{\DD(u_1,u_2)} \rd_v(r^{p-2}) \frac{|\rd_u \psi|^2}{\rd_v r}= \frac{1}{2-p}\int_{u_1}^{u_2} R^{p-2}\frac{|\rd_u \psi|^2}{\rd_v r}(u,v_R(u)) du\\ &+  \frac{1}{2-p} \int_{\DD(u_1,u_2)} r^{p-4} [r^2\rd_{v}^2 r] \frac{|\rd_u \psi|^2}{[\rd_v r]^2} \\&+\frac{2}{2-p} \left[ \int_{\DD(u_1,u_2)}\frac{\Omega^2}{4\rd_v r}r^{p-4}\slashed{\Delta}_{\mathbb{S}^2} \psi  \rd_u \psi    + \epsilon\ r^{p-4}\left(\sum_{i=0}^{1} \frac{ \tw_i(u,v) }{\rd_v r}\cdot   \rd_u^{i}\psi+  \frac{ \tq(u,v)}{\rd_v r} \cdot r \rd_v \psi \right) \rd_u \psi \right].
		\end{split}
	\end{equation*}
	Now, using \eqref{H0} and \eqref{H2}, we can absorb some of the terms using the largeness of $R$ and smallness of $\ep$ to write
		\begin{equation}\begin{split}\label{ineq1}
			&\bigl| \int_{\DD(u_1,u_2)} r^{p-3} |\rd_u \psi|^2- \frac{2}{2-p} [ \int_{\DD(u_1,u_2)}\frac{\Omega^2}{4\rd_v r}r^{p-4}\slashed{\Delta}_{\mathbb{S}^2} \psi  \rd_u \psi     \bigr| \\& \ls  [2-p]^{-1}\int_{u_1}^{u_2} R^{p-2}|\rd_u \psi|^2(u,v_R(u)) du + \ep^2 [2-p]^{-2} \int_{\DD(u_1,u_2)} r^{p-3} |\rd_v \psi|^2 + \ep^2  \int_{\DD(u_1,u_2)} r^{p-5} | \psi|^2.
		\end{split}
	\end{equation} Now, by \eqref{Hardy5} with $q=p-2$, the last term can be controlled in the following fashion: 
		\begin{equation}\begin{split}\label{ineq2}
			&\bigl| \int_{\DD(u_1,u_2)} r^{p-3} |\rd_u \psi|^2- \frac{2}{2-p}  \int_{\DD(u_1,u_2)}\frac{\Omega^2}{4\rd_v r}r^{p-4}\slashed{\Delta}_{\mathbb{S}^2} \psi  \rd_u \psi     \bigr| \\& \ls  R^{p-4}\int_{u_1}^{u_2} (R^2 [2-p]^{-1}|\rd_u \psi|^2+|\psi|^2)(u,v_R(u)) du + \ep^2 [2-p]^{-2} \int_{\DD(u_1,u_2)} r^{p-3} |\rd_v \psi|^2.
		\end{split}
	\end{equation}
	
	 Note that integrating by parts on $\mathbb{S}^2$ and then in $u$ gives \begin{equation*}\begin{split}
			& \int_{\DD(u_1,u_2)}\frac{\Omega^2}{\rd_v r}r^{p-4} [\sDelta \psi \rd_u \psi]=-\frac{1}{2} \int_{\DD(u_1,u_2)}\frac{\Omega^2}{\rd_v r}r^{p-4} \rd_u[ |\snabla \psi|^2 ]=-\frac{1}{2} \int_{C_{u_2}}  \frac{\Omega^2}{\rd_v r}r^{p-4}  |\snabla \psi|^2 + \frac{1}{2} \int_{C_{u_1}}  \frac{\Omega^2}{\rd_v r}r^{p-4}  |\snabla \psi|^2 \\ & - \frac{1}{2} \int_{v_R(u_1)}^{v_R(u_2)}  \frac{\Omega^2}{\rd_v r}R^{p-4}  |\snabla \psi|^2 (u_R(v'),v')dv'+ \frac{1}{2} \int_{\DD(u_1,u_2)}  \rd_u \left(\frac{\Omega^2}{\rd_v r}r^{p-4} \right) |\snabla \psi|^2,
		\end{split}
	\end{equation*}
	which, combining to \eqref{H0} and \eqref{ineq2} gives 
		\begin{equation}\begin{split}\label{ineq3}
			&\int_{C_{u_2}} r^{p-4}  |\snabla \psi|^2+[2-p] \int_{\DD(u_1,u_2)} r^{p-3} |\rd_u \psi|^2  \ls \int_{C_{u_1}} r^{p-4}  |\snabla \psi|^2+\\& +  R^{p-4}\int_{u_1}^{u_2} (R^2|\rd_u \psi|^2+| \psi|^2+ |\snabla \psi|^2)(u,v_R(u)) du + \ep^2 [2-p]^{-1} \int_{\DD(u_1,u_2)} r^{p-3} |\rd_v \psi|^2.
		\end{split}
	\end{equation} completing the proof of \eqref{Hardy.in}.

	\end{proof}

The next proposition records the well-known Poincar\'{e} inequality.
\begin{prop}\label{Poincare.prop} Let $f \in H^1(\mathbb{S}^2)$. Then the following inequality holds:
\begin{equation}\label{Poincare}
	\int_{\mathbb{S}^2} |\snabla f_{\geq L}|^2 \geq L(L+1)  \int_{\mathbb{S}^2} | f_{\geq L}|^2 .
\end{equation} Moreover, \begin{equation}\label{Poincare2}
\int_{\mathbb{S}^2} |\snabla f_{L}|^2 = L(L+1)  \int_{\mathbb{S}^2} | f_{ L}|^2 .
\end{equation}

\end{prop}

\begin{proof}
	This is very standard, see for instance \cite{MoiDejan}, Lemma 2.1.
\end{proof}

	\subsection{Lower-weighted energy estimates}\label{low.section} 
	We start with a first boundedness statement for $r^p$-weighted estimates for $0\leq p <2$. The argument relies on the Hardy inequalities of Proposition~\ref{Hardy.prop} and provides (non-optimal) decay results in time for the unweighted energy.
	\begin{prop}\label{lower.prop}
There exists $p_{low}(\ep) <2$ with $2-p_{low}(\ep)=O(\sqrt{\epsilon})$ such that for all $0<p \leq p_{low}$:

 \begin{equation}\begin{split}\label{low.rp.est}
		&	\int_{C_{u_2}} r^p|\rd_v \psi|^2+ p \int_{\D(u_1,u_2)}    r^{p-1} |\rd_v \psi|^2+	(2-p)  r^{p-2}|\snabla \psi|^2\\ &\lesssim     \int_{C_{u_1}}[  r^p|\rd_v \psi|^2+ r^{p-4}  |\snabla \psi|^2]+\boundary
	\end{split}
\end{equation}

	\end{prop}
	
	\begin{proof}
 
 By \eqref{psi.eq}, $\psi$ satisfies \eqref{psi.source} with $F= \ep\ r^{-2}\left(\sum_{i=0}^{1} \tw_i(u,v)\cdot   \rd_u^{i}\psi+   \tq(u,v)\cdot r \rd_v \psi \right) $, so we can apply Lemma~\ref{lemma.calc1} and integrate \eqref{rp1} on $\DD(u_1,u_2)$ to obtain \begin{equation*}\begin{split}
&\int_{C_{u_2}} r^p|\rd_v \psi|^2+ p \int_{\D(u_1,u_2)}    r^{p-1} |\rd_v \psi|^2+	(2-p)  r^{p-2}|\snabla \psi|^2\lesssim     \int_{C_{u_1}}  r^p|\rd_v \psi|^2\\&+  \ep \int_{\D(u_1,u_2)}  [  r^{p-1} |\rd_v \psi|^2+ r^{p-3}[ |\psi|^2+|\rd_u \psi|^2] + \boundary\\& \lesssim  \int_{C_{u_1}}  r^p|\rd_v \psi|^2+  \ep \int_{\D(u_1,u_2)}    r^{p-3}[ |\psi|^2+|\rd_u \psi|^2]+\boundary.
 	\end{split}
 \end{equation*} Note that in the above, we have used the bound on $F$ coming from \eqref{H2} giving \begin{equation}
 r^p |\rd_v \psi F | \ls  \ep \left( r^{p-2}( |\psi|+  |\rd_u\psi|)|\rd_v\psi|+    r^{p-1} |\rd_v \psi|^2 \right) \ls \ep  \left( r^{p-3}( |\psi|^2+  |\rd_u\psi|^2)+    r^{p-1} |\rd_v \psi|^2 \right).
 \end{equation}
Now, we apply Proposition~\ref{Hardy.prop}  and by \eqref{Hardy5}, \eqref{Hardy.in} with $q=p$, we have \begin{equation}
	\begin{split} &  \int_{\D(u_1,u_2)}    r^{p-3}[ |\psi|^2+ || \rd_u \psi|^2] \lesssim  [2-p]^{-2}\int_{\D(u_1,u_2)}   r^{p-1} |\rd_v \psi|^2+ \frac{\ep^2}{R^3 [2-p]} \int_{\D(u_1,u_2)} r^{p-2} |\snabla \psi|^2  \\ &  +[2-p]^{-1}\int_{C_{u_1}} r^{p-4}  |\snabla \psi|^2+\boundary.
		\end{split} 
\end{equation}  hence, using the smallness of $\ep$, on condition that $(2-p) \gg \sqrt{\ep}$, we get \eqref{low.rp.est}.

	\end{proof}
	
	\begin{cor} \label{energy.low.cor} For $0\leq p\leq p_{low}(\ep)$ and $u_1<u_2$:
		\begin{equation}\label{rp.low.boundedness}
	p	\int_{u_1}^{u_2} E_{p-1}[\psi](u ) du + \tilde{E}_p(u_2) \lesssim  	\tilde{E}_p(u_1).
		\end{equation}
	\end{cor}
	\begin{proof}
		Note that, since $p_{low}<2$, we have $r^{p-4} |\snabla \psi|^2 \lesssim  |\snabla \phi|^2 $, hence \begin{equation}
			\int_{C_u} r^{p-4} |\snabla \psi|^2 \lesssim E(u).
		\end{equation} Using this, together with Proposition~\ref{lower.prop} and \eqref{boundedness} immediately proves 
		\begin{equation}
			p	\int_{u_1}^{u_2} E_{p-1}[\psi](u ) du+	\tilde{E}_p(u_2) \lesssim  	\tilde{E}_p(u_1)+\boundary.
		\end{equation}
		
 Then \eqref{rp.low.boundedness} follows from an application of Lemma~\ref{lemma.average}.
	\end{proof}
	We finally obtain a first (non-optimal) decay result of the energy.
		\begin{cor} \label{low.cor.decay} 
			\begin{equation}
				E[\phi](u) \lesssim u^{-p_{low}} \lesssim u^{-2+ O(\sqrt{\ep})}.
			\end{equation}
	\end{cor}
	\begin{proof}
	Immediate from the combination of  Corollary~\ref{energy.low.cor} and Lemma~\ref{lemma.hierarchy}.
	\end{proof}

	
		\subsection{Faster energy decay for higher spherical harmonics} 
		In this section, we prove that the higher spherical harmonics decay faster than the spherical average. Our goal is to reproduce the $0\leq p \leq 2$ hierarchy for the commuted quantity $\Psi_1 = r^2 \rd_v \psi$ (see \cite{AAG0} where this idea was first introduced). First, we define \begin{equation}
		\phione(u,v,\omega) := \phi(u,v,\omega) -\int_{\mathbb{S}^2} \phi(u,v,\omega) d\omega
		\end{equation} and $\psione(u,v,\omega) = r \phione(u,v,\omega)$. We will also denote $\Psi_1(u,v,\omega)= r^2  \rd_v \psione(u,v,\omega)$.

		\begin{lemma}\label{r^2.dF.lemma}
			Recall $F= \ep\ r^{-2}\left(\sum_{i=0}^{1}\tw_i(u,v)\cdot   \rd_u^{i}\psi+   \tq(u,v)\cdot r \rd_v \psi \right) $. Then, if $\phi$ is a solution of \eqref{wave.main}: \begin{equation}\label{r^2.dF}
			|\rd_v (\frac{r^2 F}{\Omega^2})-\ep\cdot \frac{\tw_1(u,v)}{4r^2}\sDelta \psi|\lesssim \ep\ r^{-2} \left( |\psi|+|\rd_u \psi|+|\Psi_1| + r|\rd_v \Psi_1|\right)
			\end{equation}
		\end{lemma}
	\eqref{r^2.dF} is also true replacing $F$ by $F_{\geq 1}$, mutatis mutandis.
		\begin{proof}
		Immediate calculation, using \eqref{psi.eq}, \eqref{H0} and \eqref{H2}.
		\end{proof}
		
		\begin{prop}\label{prop.ang} \label{higher.prop}For all $u_1<u_2$, we have for all $0<p \leq p_{low}$:
		\begin{equation}\begin{split}\label{higher.est}
				&		\int_{\Cutwo}  r^p|\rd_v \Psi_1|^2+	\int_{\D(u_1,u_2)} 	\left( r^{p-1} |\rd_v \Psi_1|^2 + r^{p+1} |\rd_v\psi|^2\right)\\ & \lesssim 	\int_{\Cuone}  r^p|\rd_v \Psi_1|^2+    \boundarytwo. \end{split}
				\end{equation}
		\end{prop}
		\begin{proof}
Fix $0<p<2$. We make use of Lemma~\ref{lemma.calc2} and integrate \eqref{rp2} on $\D(u_1,u_2)$. By the fact that for $p<2$, the boundary terms in $v$ of  $\rd_v\left( r^{p-2}\big( -\rd_u r \rd_v r+ r \rd_{u} \rd_v r\big)|\Psi_1|^2- \frac{\Omega^2 r^{p-2}}{8} |\snabla \Psi_1|^2\right) $ on $\mathcal{I}^+=\{v=\infty\}$ are $0$ and taking advantage of Lemma~\ref{r^2.dF.lemma},  we obtain \begin{equation}\begin{split}
		&	\int_{\Cutwo}  r^p|\rd_v \Psi_1|^2+	\int_{\D(u_1,u_2)} 	[2+\frac{p}{2}] r^{p-1} |\rd_v \Psi_1|^2		 \\ &  + \int_{\D(u_1,u_2)}  r^{p-3}\left(\frac{\Omega^2}{8}\big([2-p] \rd_v r -  r\rd_v \log(\Omega^2) \big) |\snabla \Psi_1|^2 + \big(  (p-2)\rd_v r [-\rd_v r \rd_u r  + r \rd_{uv}^2 r] - (\rd_{vv}^2 r)(\rd_u r)+ r \rd_u \rd_v \rd_v r\big)|\Psi_1|^2\right) \\ &\ls	\int_{\Cuone} r^p|\rd_v \Psi_1|^2+ \ep \int_{\D(u_1,u_2)}  r^{p-2}| \rd_v \Psi_1|(|\psi|+|\rd_u \psi|+|\Psi_1| + r|\rd_v \Psi_1|)+ \ep \bigl| \int_{\D(u_1,u_2)} r^{p-2} \tw_1 [\rd_v \Psi_1] \sDelta\psi \bigr| \\&+\boundarytwo. \end{split}
\end{equation} Now, we use the smallness of $\ep$ to absorb the $\ep \int_{\D(u_1,u_2)}  r^{p-2}| \rd_v \Psi_1|(|\psi|+|\rd_u \psi|+|\Psi_1| + r|\rd_v \Psi_1|)$ term into the left-hand-side from which we obtain 
\begin{equation}\begin{split}\label{ineq4}
		&	\int_{\Cutwo}  r^p|\rd_v \Psi_1|^2+	\int_{\D(u_1,u_2)} 	[2+\frac{p}{2}] r^{p-1} |\rd_v \Psi_1|^2		 \\ &  + \int_{\D(u_1,u_2)}  r^{p-3}\left(\frac{\Omega^2}{8}\big([2-p] \rd_v r -  r\rd_v \log(\Omega^2) \big) |\snabla \Psi_1|^2 + \big(  (p-2)\rd_v r [-\rd_v r \rd_u r  + r \rd_{uv}^2 r] - (\rd_{vv}^2 r)(\rd_u r)+ r \rd_u \rd_v \rd_v r\big)|\Psi_1|^2\right) \\ &\ls	\int_{\Cuone} r^p|\rd_v \Psi_1|^2+ \ep \int_{\D(u_1,u_2)}  r^{p-3}(|\psi|^2+|\rd_u \psi|^2+|\Psi_1|^2)+ \ep \bigl| \int_{\D(u_1,u_2)} r^{p-2} \tw_1 [\rd_v \Psi_1] \sDelta\psi \bigr| \\&+\boundarytwo. \end{split}
\end{equation} 
Apply Proposition~\ref{Hardy.prop}   with $q=p$ gives, also exploiting the fact that $\int_{C_{u_1}} r^{q-4}  |\snabla \psi|^2 \lesssim E(u_1)$:
\begin{equation}\begin{split}\label{high1}
&\int_{\D(u_1,u_2)}  r^{p-3}(|\psi|^2+|\rd_u \psi|^2)\ls [2-p]^{-2}\int_{\DD(u_1,u_2)}\left[  r^{p-3} |\rd_v  \psi|^2+ r^{p-5} |\snabla \psi|^2\right]  \\  & +(2-p)^{-1}(E(u_1)+\boundary).\end{split}
\end{equation}Now we can apply Proposition~\ref{Hardy.prop} again to $\snabla \psi$ this time  with $q=p-2$ to obtain\begin{equation}\label{high2}
 \int_{\D(u_1,u_2)} r^{p-5} |\snabla \psi|^2\lesssim  \int_{\D(u_1,u_2)} r^{p-3} |\snabla \rd_v \psi|^2 + \boundarytwo.
\end{equation}
Combining \eqref{high1}, \eqref{high2}, and recalling that $r^2 |\rd_v \psi| \lesssim |\Psi_1|$, we get \begin{equation}\begin{split}\label{ineq5}
&\int_{\D(u_1,u_2)}  r^{p-3}(|\psi|^2+|\rd_u \psi|^2)\ls [2-p]^{-2}\int_{\DD(u_1,u_2)}r^{p-7}\left[   |\Psi_1|^2+ |\snabla\Psi_1|^2\right]  \\  & +(2-p)^{-1}(E(u_1)+\boundarytwo).\end{split}
\end{equation}

 We now return to \eqref{ineq4}, in particular its third term. By  \eqref{Poincare} (Poincar\'{e}'s inequality) in Proposition~\ref{Poincare.prop}, we have $\int_{\mathbb{S}^2} |\snabla (\Psi_1)_{\geq 2}|^2 \geq 6 \int_{\mathbb{S}^2} | (\Psi_1)_{\geq 2}|^2  $, and of course $\int_{\mathbb{S}^2} |\snabla (\Psi_1)_{= 1}|^2  = 2  \int_{\mathbb{S}^2} | (\Psi_1)_{=1}|^2  $, hence, as a consequence of \eqref{H0}, we have a coercive estimate of the form 
 
 \begin{equation}\begin{split}
 &	\int_{\D(u_1,u_2)}  r^{p-3}\left(\frac{\Omega^2}{8}\big([2-p] \rd_v r -  r\rd_v \log(\Omega^2) \big) |\snabla \Psi_1|^2 + \big(  (p-2)\rd_v r [-\rd_v r \rd_u r  + r \rd_{uv}^2 r] - (\rd_{vv}^2 r)(\rd_u r)+ r \rd_u \rd_v \rd_v r\big)|\Psi_1|^2\right)\\ & \gtrsim [2-p] \int_{\D(u_1,u_2)}  \left(r^{p-3} |\snabla (\Psi_1)_{\geq 2}|^2- r^{p-4} |(\Psi_1)_{= 1}|^2\right)\gtrsim  [2-p] \int_{\D(u_1,u_2)}  \left(r^{p-4} |\snabla \Psi_1|^2- r^{p-4} |\Psi_1|^2\right).  \end{split}
 \end{equation}
 Thus, combining this estimate with \eqref{ineq4}, \eqref{ineq5}  (we took $R(\ep)$ to be large enough in order to absorb various terms into the left-hand-side) leads to

 \begin{equation}\begin{split}\label{ineq6}
 		&	\int_{\Cutwo}  r^p|\rd_v \Psi_1|^2+	\int_{\D(u_1,u_2)} 	[2+\frac{p}{2}] r^{p-1} |\rd_v \Psi_1|^2+ [2-p] \int_{\D(u_1,u_2)} r^{p-4} |\snabla \Psi|^2		 \\ &\ls	\int_{\Cuone} r^p|\rd_v \Psi_1|^2+ \ep \int_{\D(u_1,u_2)}  r^{p-3}|\Psi_1|^2+ \ep \bigl| \int_{\D(u_1,u_2)} r^{p-2} \tw_1 [\rd_v \Psi_1] \sDelta\psi \bigr| \\&+E(u_1)+\boundarytwo. \end{split}
 \end{equation}  Now, by Proposition~\ref{Hardy.prop} (Hardy inequality) applying to $\Psi_1$ with $q=p$, we get, using the fact that $\ep [2-p]^{-1}=o(1)$ since $p\leq p_{low}$,
 
  \begin{equation}\begin{split}\label{ineq7}
 		&	\int_{\Cutwo}  r^p|\rd_v \Psi_1|^2+	\int_{\D(u_1,u_2)} 	[2+\frac{p}{2}] r^{p-1} |\rd_v \Psi_1|^2+ [2-p] \int_{\D(u_1,u_2)} r^{p-4} |\snabla \Psi|^2		 \\ &\ls	\int_{\Cuone} r^p|\rd_v \Psi_1|^2+ \ep \bigl| \int_{\D(u_1,u_2)} r^{p-2} \tw_1 [\rd_v \Psi_1] \sDelta\psi \bigr| +E(u_1)+\boundarytwo. \end{split}
 \end{equation} 
Finally, we handle the term $\int_{\D(u_1,u_2)} r^{p-2} \tw_1 [\rd_v \Psi_1] \sDelta\psi$.  Integrating by parts on $\mathbb{S}^2$ and in $v$ gives, using \eqref{H2} \begin{equation}\begin{split}\label{ineq8}
&\bigl| \int_{\DD(u_1,u_2)}	 r^{p-2} \rd_v \Psi_1 \tw_1(u,v) \sDelta \psi\bigr|=\bigl| \int_{\DD(u_1,u_2)}	  r^{p-2}   \tw_1(u,v) \rd_v \snabla \Psi_1 \cdot \snabla \psi\bigr|=\bigl| \int_{\DD(u_1,u_2)}	  r^{p-2}   \tw_1(u,v) \rd_v \snabla \Psi_1 \cdot \snabla \psi\bigr|\\ & \ls \int_{\DD(u_1,u_2)}	  r^{p-4}   | \snabla \Psi_1|^2+|\int_{\DD(u_1,u_2)}	r^{p-1}|\psi||\rd_v \psi| +\boundarytwo\\ &  \ls \int_{\DD(u_1,u_2)}	  r^{p-4}   | \snabla \Psi_1|^2+ r^{p-3} |\psi|^2 + r^{p-3} |\Psi_1|^2 +\boundarytwo\\ & \ls \int_{\DD(u_1,u_2)}	  r^{p-4}   | \snabla \Psi_1|^2 + r^{p-3} |\Psi_1|^2 +\boundarytwo ,
\end{split}
\end{equation} where in the last inequality, we have applied Proposition~\ref{Hardy.prop} (Hardy's inequality) to $\psi$ for $q=p$.

Now, using  Proposition~\ref{Hardy.prop}, this time on $\Psi_1$ with $q=p$ finally gives
\begin{equation}\begin{split}\label{ineq10}
		&\bigl| \int_{\DD(u_1,u_2)}	 r^{p-2} \rd_v \Psi_1 \tw_1(u,v) \sDelta \psi\bigr|\\& \ls \int_{\DD(u_1,u_2)}	  r^{p-4}   | \snabla \Psi_1|^2 + r^{p-1} |\rd_v\Psi_1|^2 +\boundarytwo .
	\end{split}
\end{equation} 
Combining \eqref{ineq7}, \eqref{ineq8} and using the smallness of $\ep$ (recall that $\ep=o(2-p)$ by choice of $p_{low}$) to absorb the right-hand-side of \eqref{ineq8} into the left-hand-side of \eqref{ineq7} finally leads to  \begin{equation}\begin{split}\label{ineq9}
		&	\int_{\Cutwo}  r^p|\rd_v \Psi_1|^2+	\int_{\D(u_1,u_2)} 	[2+\frac{p}{2}] r^{p-1} |\rd_v \Psi_1|^2+ [2-p] \int_{\D(u_1,u_2)} r^{p-4} |\snabla \Psi|^2		 \\ &\ls	\int_{\Cuone} r^p|\rd_v \Psi_1|^2 +E(u_1)+\boundarytwo, \end{split}
\end{equation} 

which concludes the proof of \eqref{higher.est}, noting again that Proposition~\ref{Hardy.prop} provides the control of 	$\int_{\D(u_1,u_2)} r^{p-3}|\Psi_1|^2=	\int_{\D(u_1,u_2)} r^{p+1}|\psi|^2$.

		\end{proof}
		We deduce a $r^p$-weighted-hierarchy analogous to that of Corollary~\ref{energy.low.cor}.
		\begin{cor}\label{energy.high.cor}
For $0\leq p\leq p_{low}(\ep)$ and $u_1<u_2$:
\begin{equation}\label{rp.high.boundedness}
		\int_{u_1}^{u_2} \left( E_{p-1}[\Psi_1](u )+ E_{p+1}[\psi](u )\right) du + E_p[\Psi_1](u_2) \lesssim  E_p[\Psi_1](u_1)+ \sum_{|\beta|\leq 1}E[\snabla^{\beta}\phi](u_1).
\end{equation}
\end{cor}	\begin{proof}
This is an immediate consequence of Proposition~\ref{higher.prop}, following the same proof as Corollary~\ref{energy.low.cor}.
\end{proof}
Finally, we apply the hierarchy of Corollary~\ref{energy.high.cor} to deduce the faster energy-decay of higher-order spherical harmonics.

\begin{cor}\label{decay.high.cor}
The following estimate holds for all $u\geq u_0$:
\begin{equation} \label{high.decay.est}
	E[\phi_{\geq 1}](u)  \lesssim u^{-2p_{low}}. \end{equation}
\end{cor}

		\begin{proof}
			 First, we choose a dyadic sequence $(u_n)_{n \in \mathbb{N}}$ and apply Corollary~\ref{rp.high.boundedness} to $u_1=u_n$, $u_2=u_{n+1}$; then, by the mean-value theorem, there exists  $u_n^{*} \in [u_n,u_{n+1}]$ such that \begin{equation*} 
			 E_{p-1}[\Psi_1](u_n^{*})  \lesssim u_n^{-1} \left(E_p[\Psi_1](u_n)+ \sum_{|\beta|\leq 1}E[\snabla^{\beta}\phi_{\geq 1}](u_n)\right)\lesssim u_n^{-1}		   \left(E_p[\Psi_1](u_0)+ \sum_{|\beta|\leq 1}E[\snabla^{\beta}\phi_{\geq 1}](u_0)\right)\ls u_n^{-1}, \end{equation*} where in the last inequality, we have also used \eqref{boundedness} applied to $\snabla^{\beta}\phi_{\geq 1}$. Then by using Corollary~\ref{rp.high.boundedness} again, we obtain that for all $u\in [u_n,u_{n+1}]$
			 \begin{equation*} 
			 	E_{p-1}[\Psi_1](u)  \lesssim u^{-1}. \end{equation*} By Lemma~\ref{lemma.hierarchy}, we obtain that  \begin{equation*} 
			 	E_{0}[\Psi_1](u)= \int_{C_u} |\rd_v \Psi_1|^2  \lesssim u^{-p_{low}}. \end{equation*} Then, by Proposition~\ref{Hardy.prop} (Hardy inequality), we obtain 
			 	\begin{equation*} 
			 	 \int_{C_u}  r^{-2}| \Psi_1|^2 =  \int_{C_u}  r^{2}| \rd_v \psi_{\geq 1}|^2= E_{2}[\psi_{\geq 1}](u)  \lesssim u^{-p_{low}}. \end{equation*}
			 	 
			 	 In particular, we have
			 	 	\begin{equation*} 
			 	  E_{p_{low}}[\psi_{\geq 1}](u)  \lesssim u^{-p_{low}}. \end{equation*}
			 	  
			 	  Then, we can apply Lemma~\ref{lemma.hierarchy} again to $\psi_{\geq 1}$,  this time invoking Corollary~\ref{energy.low.cor} applied to $\phi_{\geq 1}$ to finally obtain \eqref{high.decay.est}
			 as desired.

		\end{proof}
	We conclude this section by proving point-wise decay estimates corresponding to \eqref{pointwise.est}, \eqref{pointwise.est2}, but for $\phi_{\geq 1}$ instead of the full solution.
	
	\begin{cor}
		\label{cor.pointwise.ang} The following estimates hold: for all $\eta>0$, $u\geq u_0$
		\begin{equation}\label{pointwise.est.ang}
r|\phi_{\geq 1}|(u,v,\omega) \lesssim_{\eta}  u^{\frac{1}{2} - p_{low}+\eta}\lesssim u^{-\frac{3}{2}+ O(\sqrt{\ep})},
		\end{equation} 
			\begin{equation}\label{pointwise.est.ang2}
			r^{\frac{1}{2}}|\phi_{\geq 1}|(u,v,\omega) \lesssim_{\eta}  u^{ - p_{low}+ \eta}\lesssim u^{-2+ O(\sqrt{\ep})}.
		\end{equation}
	\end{cor}
	
	\begin{proof}
		 We write, combining the Sobolev embedding on $\mathbb{S}^2$ and  the Cauchy-Schwarz inequality in $v$: for any $\eta>0$ \begin{equation}	|\psi_{\geq 1}|(u,v,\omega)
			\lesssim \sum_{|\beta| \leq 2}	\|\snabla^{\beta}\psi_{\geq 1}(u,v,\cdot)\|_{L^2(\mathbb{S}^2)}\lesssim  \sum_{|\beta| \leq 2}\left( R	\|\snabla^{\beta}\phi_{\geq 1}(u,v_R(u),\cdot)\|_{L^2(\mathbb{S}^2)} +  [\int_{v_R(u)}^{+\infty}\int_{\mathbb{S}^2} r^{1+2\eta} |\rd_v \snabla^{\beta}\psi_{\geq 1}|^2]^{1/2}\right).
		\end{equation} 
\eqref{pointwise.est.ang}  then follows immediately from the application of Proposition~\ref{energy.decay.prop} with $\gamma= \frac{1}{2}- \eta$, and Proposition~\ref{prop.ang} applied to $\snabla^{\beta }\phi$ for $|\beta| \leq 2$.  \eqref{pointwise.est.ang2} is obtained similarly as a consequence of Proposition~\ref{energy.decay.prop} with $\gamma= \eta$.

	\end{proof}
	
	\subsection{Higher-weighted energy estimates for the spherical average} \label{high.section}
	
	In Section~\ref{low.section}, we managed to prove the boundedness of $r^p$-weighted estimates for $p\leq p_{low} = 2 - O(\sqrt{\ep})$. In this section, we want to obtain a similar result for $2\leq p < 3 -O(\sqrt{\ep})$, which comes arbitrarily close to the sharp exponents $p<3$. To do this, we will restrict to our attention to the spherical-average of the solution \begin{equation}
\phi_0 = \int_{\mathbb{S}^2} \phi(u,v,\omega) d\omega.
	\end{equation} This section generalizes the treatment of  \cite{Moi2} in Section 6.3, whose main idea is to absorb the error terms inside the boundary term $E_p[\psi](u_2)$ instead of absorbing them into the bulk term $\int_{u_1}^{u_2} E_{p-1}[\psi](u) du$ with the help of a Grönwall argument. We point out, however, that the argument has been extended and significantly streamlined compared to its previous version present in \cite{Moi2}.

	We start by quoting an easy calculus lemma, on which the Grönwall argument will be eventually based. \begin{lemma}\label{calc.lemma}The following identity holds for all $u_1<u_2$:\begin{equation}
\int_{u_1}^{u_2} \frac{du}{a \sqrt{u}+b} = \frac{2}{a} \left[ \sqrt{u}- \frac{b}{a} \log(\sqrt{u}+ \frac{b}{a})\right]_{u_1}^{u_2},			\end{equation} where  we defined $\left[ F(u)\right]_{u_1}^{u_2}:=F(u_2)-F(u_1).$	
	\end{lemma}
		\begin{proof}
		The proof is elementary, using a change of variable $x=\sqrt{u}$.
	\end{proof}
	Then, we move to the main result of this section, which quantifies faster energy decay for the spherical average than what was obtained in Section~\ref{low.section}.
	
	\begin{prop}\label{prop.high}
		There exists $2<p_{high}(\ep)\leq 1+ p_{low}<3$ such that $p_{high} = 3 -O(\sqrt{\ep})$ and  $\delta(\ep) =O(\sqrt{\ep})$ such that for all $u\geq u_0$, $0 \leq q \leq p_{high}-1$ \begin{equation}\label{high.est0}
			\tilde{E}_q[\phi_0](u) \lesssim u^{-p_{high}+2\delta+ q},
		\end{equation} 
		\begin{equation}\label{high.estp}
		\tilde{E}_{p_{high}}[\phi_0](u) \lesssim u^{2\delta}.
		\end{equation}
	
	\end{prop}

	\begin{proof}
Let $2< p \leq 1 + p_{low}<3$. Similarly to the proof of Proposition~\ref{lower.prop}, we use the multiplier $r^p \rd_v \psi$, and we apply		Lemma~\ref{lemma.calc1} to obtain, using \eqref{H0}, \eqref{H2}  \begin{equation*}\begin{split}
		&\bigl| \int_{C_{u_2}} r^p|\rd_v \psi_0|^2-   \int_{C_{u_1}}  r^p|\rd_v \psi_0|^2\bigr| +  \int_{\D(u_1,u_2)}    r^{p-1} |\rd_v \psi_0|^2 \lesssim      \boundaryzero \\ &+ \ep \int_{\D(u_1,u_2)}    r^{p-1} |\rd_v \psi_0|^2 + \ep  \left[  \int_{\D(u_1,u_2)}    r^{p} |\rd_v \psi_0|^2\right]^{1/2} \left[\int_{\D(u_1,u_2)}    r^{p-4} (|\psi_0|^2+|\rd_u \psi_0|^2)\right] ^{1/2}
	\end{split}
\end{equation*} Then, by Proposition~\ref{Hardy.prop} applying with $q=p-1$, and the fact that $p<3$, we obtain
\begin{equation*}\begin{split}
		&\bigl| \int_{C_{u_2}} r^p|\rd_v \psi_0|^2-   \int_{C_{u_1}}  r^p|\rd_v \psi_0|^2\bigr| +  \int_{\D(u_1,u_2)}    r^{p-1} |\rd_v \psi_0|^2\lesssim      \boundaryzero \\ & + \ep [3-p]^{-1}  \left[  \int_{\D(u_1,u_2)}    r^{p} |\rd_v \psi_0|^2\right]^{1/2} \left[\int_{\D(u_1,u_2)}    r^{p-2}   |\rd_v \psi_0|^2\right] ^{1/2}.
	\end{split}
\end{equation*} Then, by an averaging argument in $R$ similar to the one used in Corollary~\ref{energy.low.cor},
\begin{equation}\label{energy.ineq}\begin{split}
&  \int_{u_1}^{u_2} E_{p-1}[\psi_0](u) du+\bigl| E_p[\psi_0](u_2)-E_p[\psi_0](u_1)\bigr| \\ &\ls \bigl| E_p[\psi_0](u_2)-E_p[\psi_0](u_1)\bigr|  \ls     \ep [3-p]^{-1} \left[\int_{u_1}^{u_2} E_p[\psi_0](u)du \right]^{\frac{1}{2}}\left[ \int_{u_1}^{u_2}E_{p-2}[\psi_0] du \right]^{\frac{1}{2}}+  E(u_1).\end{split}
\end{equation} Then, applying Corollary~\ref{energy.low.cor}, in view of the fact that $p-1 \leq p_{low}$, we obtain \begin{equation}
\label{energy.ineq2}\begin{split}
	&  \int_{u_1}^{u_2} E_{p-1}[\psi_0](u) du+\bigl| E_p[\psi_0](u_2)-E_p[\psi_0](u_1)\bigr|\\& \ls \bigl| E_p[\psi_0](u_2)-E_p[\psi_0](u_1)\bigr|  \ls     \ep [3-p]^{-1} \left[\int_{u_1}^{u_2} E_p[\psi_0](u)du \right]^{\frac{1}{2}}\left[ \tilde{E}_{p-1}(u_1) \right]^{\frac{1}{2}}+  E(u_1).\end{split}
\end{equation} 
In what follows, we take $(u_n)_{n \in \mathbb{N}}$, a dyadic sequence, and $u_2=u\in [u_n,u_{n+1}]$ and $u_1=u_{n}$.
For some large $\Delta>0$, $D>0$ and small $\delta\in(0,\frac{1}{2})$ to be determined later, we  introduce the induction hypotheses  \begin{equation}\label{I}
	\tilde{E}_{p-1}(u_k) \leq D^2 \cdot \Delta^2  u_k^{-1+ 2 \delta},
\end{equation}\begin{equation}\label{II}
\tilde{E}_{p}(u_k) \leq  \Delta^2  u_k^{ 2 \delta}.
\end{equation} \eqref{I}, \eqref{II} are obviously satisfied for $k=0$ (providing $\Delta$ is large enough, depending on $u_0$, where we recall $u_0>1$) and we will assume it holds for all $0\leq k \leq n$. Then, combining \eqref{I} and \eqref{energy.ineq2} shows that there exists $C>0$ (independent of $\ep$ and $p$) and $\eta(\ep,p)= C \ep [3-p]^{-1}$ such that
\begin{equation}\label{g2}
 E_p[\psi_0](u) \leq       \underbrace{ \Delta\cdot D\cdot \eta(\ep)\cdot u_n^{-\frac{1}{2}+\delta}}_{=a(u_n,\Delta,\ep)} \left[\int_{u_n}^{u} E_p[\psi_0](u')du' \right]^{\frac{1}{2}}+\underbrace{E_p[\psi_0](u_n)+ C   E(u_n)}_{=b(u_n)}.
\end{equation} 
Therefore, as an application of Lemma~\ref{calc.lemma}
	\begin{equation}
	\left(\int_{u_n}^{u} E_{p}[\psi_0](u')du' \right)^{\frac{1}{2}} \leq \frac{a}{2}(u-u_n)+\frac{b}{a}\log \left(1+\frac{a}{b}(\int_{u_n}^{u} E_{p}[\psi_0](u')du' )^{\frac{1}{2}}\right).
\end{equation}  Denoting $X=  \frac{a}{b}\left(\int_{u_n}^{u} E_{p}[\psi_0](u')du' \right)^{\frac{1}{2}} $ and $F(x) = x-\ln(1+x)$, a monotonically increasing function of (positive) inverse $F^{-1}(y)$. 
Thus, we have \begin{equation}
F(X) \leq \frac{a^2}{2b} [u-u_n] \leq \frac{\Delta^2 D^2 \eta^2}{2}\ \frac{u_n^{2\delta}}{b(u_n)},
\end{equation} from which we get  $X \leq F^{-1}\left(\frac{\Delta^2 \eta^2}{2}\ \frac{u_n^{2\delta}}{b(u_n)}\right)$, which translates into \begin{equation}\label{lambert}
a \left(\int_{u_n}^{u} E_{p}[\psi_0](u')du' \right)^{\frac{1}{2}} \leq b\cdot F^{-1}\left(\frac{\Delta^2  D^2\eta^2}{2}\ \frac{u_n^{2\delta}}{b(u_n)}\right).
\end{equation} Note that $F^{-1}(y) \leq 2\sqrt{y}(1+\sqrt{y})$ for all $y\geq 0$ (this follows easily from the inequality  $z\leq 2\sqrt{z}(1+\sqrt{z})-\ln(1+2\sqrt{z}(1+\sqrt{z}))$, for all $z\geq 0$), hence \eqref{lambert}, combined with \eqref{g2} gives \begin{equation}\label{lambert2}
 E_p[\psi_0](u) \leq a \left(\int_{u_n}^{u} E_{p}[\psi_0](u')du' \right)^{\frac{1}{2}}+b \leq  u_n^{\delta} \cdot\Delta\cdot D \cdot \eta(\ep) \cdot \left[ b^{\frac{1}{2}}(u_n)\sqrt{2}+ \Delta\cdot D \cdot\eta(\ep)\cdot u_n^{\delta}\right]+b(u_n).
\end{equation} 

Now, note that, as a consequence of \eqref{II} and Corollary~\ref{low.cor.decay}, there exists $C'>0$ independent of $\ep$ such that \begin{equation}\label{II.cons}
b(u_n) \leq \Delta^2 u_n^{2\delta} + C' \cdot u_n^{-p_{low}}
\end{equation}
So, combining \eqref{lambert2} and \eqref{II.cons} gives, for $u$ large enough and $D\eta(\ep)$ small enough
 \begin{equation}\label{lambert3}
	E_p[\psi_0](u)  \leq  u_n^{2\delta} \cdot\Delta^2\cdot\left(1+ 2
	\Delta \eta(\ep)\right)\leq  u_{n+1}^{2\delta} \cdot\Delta^2\cdot \eta(\ep) \cdot  \frac{1+ 2
		\Delta \eta(\ep)}{2^{2\delta}}.
\end{equation} To achieve a small $D\eta(\ep)$, we take $D$ to be independent of $\ep$ and $3-p \gtrsim \sqrt{\ep}$, so that $\eta(\ep) = O(\sqrt{\ep})$. We will also take $\Delta$ to be independent of $\ep$ and choose $\delta = O(\sqrt{\ep})$ so that $$ 2^{2\delta} = [1+2\Delta \eta(\ep)]^2$$.

In view of Corollary~\ref{low.cor.decay},  the left-hand-side of \eqref{lambert3} also controls $\tilde{E}_p$ and, therefore, \eqref{II} is satisfied for $k=n+1$, after taking $\ep$ (hence $\eta(\ep)$) small enough. Now, to prove \eqref{I} for $k=n+1$, we come back to \eqref{energy.ineq2}, which we combine with \eqref{lambert3} to obtain \begin{equation}\label{lambert4}
\int_{u_n}^{u} E_{p-1}[\psi_0](u') du'+ E_p[\psi_0](u)\ls \Delta^2\cdot \eta(\ep)\cdot u_n^{2\delta}.
\end{equation}
Then, by the mean-value theorem, there exists $u_n^{*} \in [u_n,u_{n+1}]$ such that \begin{equation}
	E_{p-1}[\psi_0](u_n^{*}) \lesssim \Delta^2\cdot \eta(\ep)\cdot u_n^{-1+2\delta},
\end{equation} and by Corollary~\ref{energy.low.cor}, we deduce that for all $u\in [u_n,u_{n+1}]$,
\begin{equation}
	E_{p-1}[\psi_0](u) \lesssim \Delta^2\cdot \eta(\ep)\cdot u^{-1+2\delta},
\end{equation} thus, \eqref{II} is satisfied for $k=n+1$ if $\ep$ (hence $\eta(\ep)$) is small enough and $D$ is large enough. Thus, \eqref{high.estp} and \eqref{high.est0} for $q=p_{high}-1$ are proved. In view of the fact that $p_{high}-1 \in (1,2)$, we can finally apply Lemma~\ref{lemma.hierarchy} and deduce \eqref{high.est0} for any $0 \leq q \leq p_{high}-1$.


		
		\end{proof}
		Then, we deduce a corollary From Proposition~\ref{prop.high} that will end-up being useful in the next section.\begin{cor}\label{cor.high}
			Let $0\leq q\leq 3+ \eta_0$, for some $\eta_0 \in (0,1)$. Then for all $u\geq u_0$, \begin{equation}
				\int_{\DD(u,\infty)} r^{q-6} \left[ |\psi_0|^2+ |\rd_u \psi_0|^2 + r^2 |\rd_v \psi_0|^2\right] \lesssim u^{-p_{high} +2\delta+ \eta_0}.
			\end{equation}
		\end{cor}\begin{proof}
			By Proposition~\ref{Hardy.prop} (Hardy inequality) and Lemma~\ref{lemma.average}, we have 
			\begin{equation}\begin{split}
					&	\int_{\DD(u,\infty)} r^{q-6} \left[ |\psi_0|^2+ |\rd_u \psi_0|^2+r^2 |\rd_v \psi_0|^2 \right] \lesssim \int_{\DD(u,\infty)} r^{q-4} |\rd_v \psi_0|^2 + \int_u^{+\infty} \left[|\psi_0|^2+ |\rd_u \psi_0|^2\right](u',v_R(u')) du'\\ &  \lesssim \int_{\DD(u,\infty)} r^{-1+ \eta_0} |\rd_v \psi_0|^2 + E(u) \lesssim \int_u^{+\infty} E_{\eta_0-1}[\psi_0](u') du' + E(u).\end{split} 
			\end{equation} Now by Corollary~\ref{energy.low.cor} and Proposition~\ref{prop.high}, we have $$\int_u^{+\infty} E_{\eta_0-1}[\psi_0](u') du' \lesssim  \tilde{E}_{\eta_0}(u) \lesssim u^{-p_{high} + 2\delta +\eta_0}, $$ which concludes the proof. 
		\end{proof}
		
	Finally,	from Proposition~\ref{prop.high}, we deduce the proof of \eqref{pointwise.est} when restrict to $\phi_0$, the spherical average.
		\begin{cor}\label{cor.rad.pointwise}
The following estimate holds: for all $u\geq u_0$, $v\geq v_R(u)$ \begin{equation}\label{pointwise.est.0}
	r |\phi_0|(u,v) \lesssim u^{\frac{1-p_{high}}{2} + \frac{\delta(\ep)}{2}}.
\end{equation}
		\end{cor}\begin{proof}   We write, by the Cauchy-Schwarz inequality for any $\eta>0$ \begin{equation}
			|\psi_0|(u,v) \lesssim_{\eta}  R	|\phi_0|(u,v_R(u)) +  [\int_{v_R(u)}^{+\infty} r^{1+\eta} |\rd_v \psi_0|^2]^{1/2}.
		\end{equation}
 Then, combining Proposition~\ref{energy.decay.prop} with $\gamma= \eta$ and Proposition~\ref{prop.high}, to  \begin{equation}
|\phi_0|(u,v_R(u)) \lesssim u^{-\frac{p_{high}}{2} + \delta+ \eta}\lesssim u^{-\frac{3}{2} + \eta+ O(\sqrt{\ep})}.
\end{equation} By Proposition~\ref{prop.high}, we also obtain 
\begin{equation}
\int_{v_R(u)}^{+\infty} r^{1+\eta} |\rd_v \psi_0|^2(u,v) dv \lesssim u^{1+\eta - p_{high}}\lesssim u^{-2 + \eta + O(\sqrt{\ep})}
\end{equation} Combining the two concludes the proof of \eqref{pointwise.est.0}, choosing $\eta=\delta=O(\sqrt{\ep})$.
		\end{proof}
		
		\subsection{Energy and radiation field pointwise decay statements  of Theorem~\ref{main.thm}}
		
		Finally, we obtain the proof of \eqref{energy.est} and \eqref{pointwise.est} in Theorem~\ref{main.thm}. First, \eqref{energy.est} results from  an immediate application of Proposition~\ref{prop.ang} and Proposition~\ref{prop.high}, noting that \begin{equation}
			E[\phi](u) =  	E[\phi_0](u)+ 	E[\phi_{\geq 1}](u)  \lesssim u^{-2p_{low}}+ u^{-p_{high} + 2\delta} \lesssim u^{-3+ O(\sqrt{\ep})}.
		\end{equation}
		Then, \eqref{pointwise.est} is obtained as an immediate application of Corollary~\ref{cor.rad.pointwise} and  Corollary~\ref{cor.pointwise.ang}.
	
		\section{Commuted energy  and point-wise decay under extra assumptions}
		In this section, we will utilize \eqref{boundedness.T}, \eqref{ILED.T}, \eqref{H4} to derive faster decay for the $T$-commuted energy \eqref{energy.est.T0} and sharp point-wise decay in the bounded-$r$ region \eqref{pointwise.est2}.	We generalize the approach of Gajic in \cite{inverse.Dejan}, that consists in establishing a new hierarchy for the quantity $\Thetazero= r \rd_v \psi$.
		\subsection{An additional hierarchy}\label{additional.section}
		
		We start to derive a $r^p$ weighted hierarchy for  $\Thetazero$ and $2<p<3$. 
		
		\begin{prop}\label{new.hierarchy.prop} There exists $C>0$ independent of $\ep$ so that for all $2+ C \ep \leq p \leq p_{high}$:
\begin{equation}\label{new.hierarchy}
	\int_{C_{u}} r^p|\rd_v \Thetazero|^2 \lesssim  u^{-p_{high}+p+ 2\delta(\ep)}.
\end{equation} Moreover, let $(u_n)_{n \in \mathbb{N}}$ be a dyadic sequence and $u_n \leq u_1 < u_2 \leq u_{n+1}$. Then 
\begin{equation}\label{new.hierarchy2}
 \int_{u_1}^{u_2}   \int_{C_{u'}} [ r^{p-1} |\rd_v \Thetazero|^2+ r^{p-3} | \Thetazero|^2] du' \lesssim  u_n^{-p_{high}+p+ 2\delta(\ep)}.
\end{equation}
		\end{prop}
		
		\begin{proof}

		We define $\Thetazero=r\rd_v \psi_0$, and assuming $\rd_u \rd_v \psi_0 = F$,  we find that \begin{equation}
	\rd_u \rd_v \Thetazero+\frac{-\rd_u r}{r}\rd_v \Thetazero+ (-r^{-1}\rd_u \rd_v r+ r^{-2} \rd_u r \rd_v r )\Thetazero	=	 \rd_v(r F).
		\end{equation} By \eqref{r^2.dF.lemma} and \eqref{H0}, we have, choosing $R(\ep)$ large enough: 
			\begin{equation}\label{theta.eq}
			\bigl|\rd_u \rd_v \Thetazero+\frac{1}{r}\rd_v \Thetazero-r^{-2}\Thetazero\bigr| \lesssim  \frac{\ep}{r^{2}}\left( |\psi_0| + |\rd_u \psi_0| + |\Thetazero|+ r|\rd_v\Thetazero|\right).
		\end{equation}

		Now, multiply  \eqref{theta.eq} by $r^p \rd_v \Thetazero$ and integrate on $\DD(u_1,u_2)$,  in the same fashion as in Lemma~\ref{lemma.calc1} and Proposition~\ref{lower.prop}, with  additionally an integration by parts in $v$ of the $r^{p-2}[\rd_v \Thetazero]\Thetazero$ term  to obtain\begin{equation*}\begin{split}
				&\int_{C_{u_2}} r^p|\rd_v \Thetazero|^2+ (p+1+O(\ep))\int_{\D(u_1,u_2)}    r^{p-1} |\rd_v \Thetazero|^2+ (p-2+O(\ep))\int_{\D(u_1,u_2)}    r^{p-3} | \Thetazero|^2\lesssim     \int_{C_{u_1}}  r^p|\rd_v \Thetazero|^2\\&+  \ep \left[\int_{\D(u_1,u_2)}   r^{p-4}[ |\psi_0|^2+|\rd_u \psi_0|^2]\right]^{\frac{1}{2}} \left[\int_{\D(u_1,u_2)}   r^{p}|\rd_v \Thetazero|^2\right]^{\frac{1}{2}} + \boundaryzero,
			\end{split}
		\end{equation*}after absorbing various terms into the left-hand-side. In what follows, we  will take $2+ C \ep <p<3$ for some $C>0$ independent of $\ep$ arranged so that all  the terms in the left-hand-side are coercive. We also use Proposition~\ref{Hardy.prop} applied with $q=p-1$ to proceed as in the proof of Proposition~\ref{prop.high} to obtain, noting the fact that $r^{p-2} |\rd_v \psi_0|^2= r^{p-4} \Thetazero^2$ and we also use  Lemma~\ref{lemma.average} to control the term $\boundaryzero$ 
		\begin{equation}\begin{split}\label{HH}
				&\int_{C_{u_2}} r^p|\rd_v \Thetazero|^2+ (p+1+O(\ep))\int_{\D(u_1,u_2)}    r^{p-1} |\rd_v \Thetazero|^2+ (p-2+O(\ep))\int_{\D(u_1,u_2)}    r^{p-3} | \Thetazero|^2\lesssim     \int_{C_{u_1}}  r^p|\rd_v \Thetazero|^2\\&+  \ep[3-p]^{-1} \left[\int_{\D(u_1,u_2)}   r^{p-2} |\rd_v\psi_0|^2\right]^{\frac{1}{2}} \left[\int_{\D(u_1,u_2)}   r^{p}|\rd_v \Thetazero|^2\right]^{\frac{1}{2}} + \tilde{E}(u_1),
			\end{split}
		\end{equation} where in the last line, we have used the fact that $ r^{p-4} |\Thetazero|^2= r^{p-2} |\rd_v\psi_0|^2$. Then, we can apply Proposition~\ref{prop.high} and take $p=p_{high}$ in \eqref{HH}, resulting in 
		\begin{equation}\begin{split}\label{HH2}
				&\int_{C_{u_2}} r^{p_{high}}|\rd_v \Thetazero|^2+ \int_{\D(u_1,u_2)}    r^{p_{high}-1} |\rd_v \Thetazero|^2+ \int_{\D(u_1,u_2)}    r^{p_{high}-3} | \Thetazero|^2\lesssim     \int_{C_{u_1}}  r^{p_{high}}|\rd_v \Thetazero|^2\\&+  \ep[3-p_{high}]^{-1} u_1^{-\frac{1}{2}+\delta} \left[\int_{\D(u_1,u_2)}   r^{p_{high}}|\rd_v \Thetazero|^2\right]^{\frac{1}{2}} + \tilde{E}(u_1),
			\end{split}
		\end{equation}
		
		We can then repeat the proof of Proposition~\ref{prop.high}  to obtain
		\begin{equation}\begin{split}\label{HH2phigh}
			&\int_{C_{u_2}} r^{p_{high}}|\rd_v \Thetazero|^2+ \int_{\D(u_1,u_2)}    r^{p_{high}-1} |\rd_v \Thetazero|^2+ \int_{\D(u_1,u_2)}    r^{p_{high}-3} | \Thetazero|^2\lesssim     u_2^{2\delta(\ep)}.
		\end{split}
		\end{equation}

		Now let $2+C\ep <p< p_{high}$. We write, using \eqref{HH2phigh} \begin{equation*}\begin{split}
				&\int_{C_{u_2}} r^p|\rd_v \Thetazero|^2+\int_{\D(u_1,u_2)}    r^{p-1} |\rd_v \Thetazero|^2+ \int_{\D(u_1,u_2)}    r^{p-3} | \Thetazero|^2\lesssim     \int_{C_{u_1}}  r^p|\rd_v \Thetazero|^2\\&+  \ep \left[\int_{\D(u_1,u_2)}   r^{2p-3-p_{high}}[ |\psi_0|^2+|\rd_u \psi_0|^2]\right]^{\frac{1}{2}} \left[\int_{\D(u_1,u_2)}   r^{p_{high-1}}|\rd_v \Thetazero|^2\right]^{\frac{1}{2}} + E(u_1)\\ & \ls  \int_{C_{u_1}}  r^p|\rd_v \Thetazero|^2+  \ep \left[\int_{\D(u_1,u_2)}   r^{2p-3-p_{high}}[ |\psi_0|^2+|\rd_u \psi_0|^2]\right]^{\frac{1}{2}} u_2^{\delta}+ E(u_1).
			\end{split}
		\end{equation*} Now, note that $2p-3-p_{high}<-1$ so we can still apply Proposition~\ref{Hardy.prop} (Hardy inequality) and obtain\begin{equation}\label{eqint}\begin{split}
			&\int_{C_{u_2}} r^p|\rd_v \Thetazero|^2+\int_{\D(u_1,u_2)}    r^{p-1} |\rd_v \Thetazero|^2+ \int_{\D(u_1,u_2)}    r^{p-3} | \Thetazero|^2\lesssim     \int_{C_{u_1}}  r^p|\rd_v \Thetazero|^2\\ & \ls  \int_{C_{u_1}}  r^p|\rd_v \Thetazero|^2+  \ep \left[\int_{\D(u_1,u_2)}   r^{2p-1-p_{high}} |\rd_v\psi_0|^2\right]^{\frac{1}{2}} u_2^{\delta}+ E(u_1)\ls  \int_{C_{u_1}}  r^p|\rd_v \Thetazero|^2  +  u_2^{\delta} \cdot u_1^{-p_{high}+p},
		\end{split}
		\end{equation} where in the last line, we have used Proposition~\ref{prop.high}. Now,  applying \eqref{eqint} if $u_1<u_2= u \in [u_n,u_{n+1}]$, where $(u_n)_{n \in \mathbb{N}}$ is a dyadic sequence easily gives \begin{equation}\label{eqint2}\begin{split}
			&\int_{C_{u}} r^p|\rd_v \Thetazero|^2\ls  \int_{C_{u_1}}  r^p|\rd_v \Thetazero|^2+     u_n^{-p_{high}+p+\delta}.
		\end{split}
		\end{equation} Note that by \eqref{HH2phigh} applied  to  $u_1=u_{n}$ and $u_2= u_{n+1}$, there exists $u_n^{*} \in [u_n, u_{n+1}]$ such that  \begin{equation}\label{eqint3}
		 \int_{C_{u_{n}^{*}}} r^{p_{high}-1}|\rd_v \Thetazero|^2\lesssim u_n^{-1+2\delta}.
		\end{equation}	By interpolation between \eqref{HH2phigh} with $u_2=u_n^{*}$ and \eqref{eqint3}, we obtain
		\begin{equation}\label{eqint4}
			\int_{C_{u_{n}^{*}}} r^{p}|\rd_v \Thetazero|^2\lesssim u_n^{-p_{high}+p+2\delta}.
		\end{equation} Taking $u_1= u_n^{*}$  in \eqref{eqint2} and combining with \eqref{eqint4} then concludes the proof of the proposition.

				\end{proof}


		\subsection{Improved decay for the time-derivative}
		In this section, we take advantage of Proposition~\ref{new.hierarchy.prop} to derive improved decay for the $T$ derivative of $\psi_0$, in order to prove \eqref{energy.est.T0} in Theorem~\ref{main.thm}.
			To fix the notations, recall $T\psi_0= \rd_u \psi_0 + \rd_v \psi_0$.  Assuming $\rd_u \rd_v \psi_0 = F$, we obtain
			\begin{equation}\label{T.psi.eq}
				\rd_v T\psi_0 = F + \rd_{v}^2 \psi_0 = F + \rd_v (r^{-1} \Thetazero)=  F + r^{-1}\rd_v  \Thetazero-  r^{-2}[\rd_v r]  \Thetazero.
			\end{equation}
			
			We start by showing improved decay for the $r^p$-weighted energy of $T\phi_0$ along a dyadic sequence. The following result -- Proposition~\ref{T.prelim.prop} below -- does not require the use of the extra assumptions \eqref{boundedness.T}, \eqref{ILED.T}, \eqref{H4}.
		\begin{prop}\label{T.prelim.prop}  There exists $p_T(\ep)\in (2,p_{high})$, $\eta_{T}(\ep)\in (0,1)$, with $p_T(\ep)= 3-O(\sqrt{\ep})$ and $\eta_T(\ep) =O(\sqrt{\ep})$ as $\ep \rightarrow 0$, such that  for all dyadic sequences $(u_n)_{n \in \mathbb{N}}$, there exists $u_n^{*} \in [u_n,u_{n+1}]$ such that
\begin{equation}\label{T.est.dyadic}
\tilde{E}_{p_T}[T\phi_0](u_n^{*}) \lesssim u_n^{-2+\eta_T(\ep)}.
\end{equation} Moreover,  the following estimate holds  for all dyadic sequences $(u_n)_{n \in \mathbb{N}}$	\begin{equation}\label{T.est.dyadic.int}
\int_{\DD(u_n,u_{n+1})}	r^{p_T}|	\rd_v T\psi_0|^2=\int_{u_n}^{u_{n+1}} \tilde{E}_{p_T}[T\phi_0](u) du \lesssim u_n^{-1+2\eta_T(\ep)}.
\end{equation} 
		\end{prop}
		\begin{proof}

	 Let $p'=3-\nu$, where $\nu \in (0,1)$. We write, using \eqref{H0}
			\begin{equation}\label{formula}
		r^{p'}|	\rd_v T\psi_0| \lesssim r^{p'}F^2+ r^{p'-4}|  \Thetazero|^2+r^{p'-2}|\rd_v  \Thetazero|^2 \lesssim \ep^2 r^{p'-4} \left(|\psi_0|^2 + |\rd_u \psi_0|^2\right)  + r^{p'-4}|  \Thetazero|^2+r^{p'-2}|\rd_v  \Thetazero|^2.
		\end{equation}  for $u_1<u_2$, we integrate \eqref{formula} on $\DD(u_1,u_2)$, and, making use of Proposition~\ref{Hardy.prop} (Hardy inequality) with $q=p'-1$, we obtain 
		\begin{equation}
		\int_{\DD(u_1,u_2)}	r^{p'}|	\rd_v T\psi_0|^2 \lesssim \int_{\DD(u_1,u_2)}  r^{p'-4} |\Thetazero|^2+ r^{p'-2}|\rd_v  \Thetazero|^2+\boundaryzero,
		\end{equation}  which turns into the following estimate after applying Lemma~\ref{lemma.average}
		\begin{equation}\label{formula2}
			\int_{\DD(u_1,u_2)}	r^{p'}|	\rd_v T\psi_0|^2 \lesssim \int_{\DD(u_1,u_2)}  r^{p'-4} |\Thetazero|^2+ r^{p'-2}|\rd_v  \Thetazero|^2+E(u_1).
		\end{equation} Now, if $\nu \geq C \ep$, then $p'-1+2\nu \in (2+C \ep, p_{high})$ and thus we have by Proposition~\ref{new.hierarchy.prop} applied to $p= p'-1+2\nu$ that, assuming as in its statement that $u_n \leq u_1 < u_2 \leq u_{n+1}$, where $(u_n)_{n \in \mathbb{N}}$ is a dyadic sequence   \begin{equation}
		\int_{\DD(u_1,u_2)}  r^{p'-4+ 2\eta} |\Thetazero|^2+ r^{p'-2+2\eta}|\rd_v  \Thetazero|^2 \lesssim u_n^{p'-p_{high} -1 + 2\delta(\ep) + 2\eta(\ep) } \lesssim u_n^{-1+ O(\sqrt{\ep})}.
		\end{equation}  Thus, combining with \eqref{formula2} and \eqref{energy.est}, we end up with

			\begin{equation}\label{formula3}
			\int_{\DD(u_1,u_2)}	r^{p'}|	\rd_v T\psi_0|^2 \lesssim u_n^{-1+O(\sqrt{\ep})}.
		\end{equation}  Thus, by the mean-value theorem, there exists $u_n^{*} \in [u_n,u_{n+1}]$ such that 
			\begin{equation}\label{formula4}
		\int_{C_{u_n^{*}}}	r^{p'}|	\rd_v T\psi_0|^2 \lesssim u_n^{-2+O(\sqrt{\ep})},
		\end{equation}  which concludes the proof.
			\end{proof}
			
			Now, we move on to using the extra assumptions \eqref{boundedness.T}, \eqref{ILED.T}, \eqref{H4} to show that the improved decay proved in Proposition~\ref{T.prelim.prop} on a dyadic sequence in fact holds for all  $u$.  Let us denote \begin{equation}\label{Tpsi.wave}
				\rd_u \rd_v T \psi_0 =  \frac{\ep}{r^2}\left(\sum_{i=0}^{1}
				\tw_i(u,v) \cdot   \rd_u^{i}T\psi_0+ 
				\tq(u,v) \cdot r \rd_v T\psi_0 \right) +\ep  F_0.
			\end{equation} As a consequence of the assumptions \eqref{H0}, \eqref{H2}, \eqref{H4} note that \begin{equation}\label{F0.est}
|F_0| \lesssim   r^{-3} (\sum_{i=0}^{1}
( |Tr||\tw_i|+  r|T\tw_i|) \cdot   |\rd_u^{i}\psi_0|+ 
( |Tr|q|+  r|Tq|) \cdot r |\rd_v \psi_0|)\ls r^{-3} (\sum_{i=0}^{1}
   |\rd_u^{i}\psi_0|+ 
 r |\rd_v \psi_0|).
			\end{equation}
		\begin{prop}
\label{T.prop} Assume \eqref{boundedness.T}, \eqref{ILED.T}, \eqref{H4} are satisfied. Then, there exists $p_{T}(\ep)\in (2,p_{high})$, $\eta_{T}(\ep)\in (0,1)$, with $p_T(\ep)= 3-O(\sqrt{\ep})$ and $\eta_T(\ep) =O(\sqrt{\ep})$ as $\ep \rightarrow 0$, such that 
\begin{equation}\label{T.est}
	\tilde{E}_{p_T}[T\phi_0](u) \lesssim u^{-2+\eta_T(\ep)}.
\end{equation} Moreover, for all $0 \leq q \leq p_T-1$,  \begin{equation}\label{T.est2}
\tilde{E}_{q}[T\phi_0](u) \lesssim u^{-2-p_T+\eta_T(\ep)+q}.
\end{equation}
		\end{prop}\begin{proof} Take $(u_n)_{n \in \mathbb{N}}$, a dyadic sequence and $u_n \leq u_1 < u_2 \leq u_{n+1}$.We will start to show that the analogue of Proposition~\ref{lower.prop} holds for $T\phi_0$. We proceed as in the proof of Proposition~\ref{lower.prop}, take $0\leq q \leq p_{low}$ 	 and  apply the multiplier $r^q T\psi_0$ to get,  as in Corollary~\ref{energy.low.cor} \begin{equation}\label{imp1}
\int_{u_1}^{u_2} E_{q-1}[T\psi_0](u) du+	\tilde{E}_q[T\psi_0](u_2)\lesssim 	\tilde{E}_q[T\psi_0](u_1)+ \ep^2 \int_{\DD(u_1,u_2)} r^{q+1} |F_0|^2+ u^{-5+O(\sqrt{\ep})},
	\end{equation}  where we have also used \eqref{boundedness.T} and Proposition~\ref{prop.high}.
Then, as in the proof of Proposition~\ref{prop.high}, we apply the multiplier $r^p T\psi_0$ with $p=p_T$, and integrate to get, using \eqref{imp1} \begin{equation}
\int_{u_1}^{u_2} E_{p-1}[T\psi_0](u) du+	E_p[T\psi_0](u_2)\lesssim 	E_p[T\psi_0](u_1)+ \ep  (E_{p-1}[T\psi_0](u_1)+ \int_{\DD(u_1,u_2)} r^p |F_0|^2)^{\frac{1}{2}}\left[\int_{u_1}^{u_2} E_p[T\psi_0](u) du\right]^{\frac{1}{2}}.
\end{equation}
Now, we can combine \eqref{F0.est} and Corollary~\ref{cor.high} with $\eta_0=O(\sqrt{\ep})$, to get \begin{equation}
	\int_{u_1}^{u_2} E_{p-1}[T\psi_0](u) du+	E_p[T\psi_0](u_2)\lesssim 	E_p[T\psi_0](u_1)+ \ep  (E_{p-1}[T\psi_0](u_1)+ u_n^{-3+O(\sqrt{\ep})})^{\frac{1}{2}}\left[\int_{u_1}^{u_2} E_p[T\psi_0](u) du\right]^{\frac{1}{2}}.
\end{equation}
		By \eqref{T.est.dyadic.int} in Proposition~\ref{T.prelim.prop}, we have \begin{equation}
			\int_{u_1}^{u_2} E_{p-1}[T\psi_0](u) du+	E_p[T\psi_0](u_2)\lesssim 	E_p[T\psi_0](u_1)+ \ep  (E_{p-1}[T\psi_0](u_1)+u_n^{-3+O(\sqrt{\ep})})^{\frac{1}{2}} u_n^{-\frac{1}{2}+ \eta_T(\ep)}.
		\end{equation} Then, take $u_1=u_{n}^{*}$ so that \eqref{T.est.dyadic} in Proposition~\ref{T.prelim.prop} is valid, thus for all $u_2 \in [u_n^{*},u_{n+1}]$\begin{equation}
		\int_{u_n^{*}}^{u_2} E_{p-1}[T\psi_0](u) du+	E_p[T\psi_0](u_2)\lesssim 	u_n^{-2+\eta_T}+ \ep  (E_{p-1}[T\psi_0](u_n^{*}))^{\frac{1}{2}} u_n^{-\frac{1}{2}+ \eta_T(\ep)}.
		\end{equation} Now, by the mean-value theorem again, there exists $u_n^{**} \in [u_n^{*},u_{n+1}]$ such that 
		\begin{equation}
		 E_{p-1}[T\psi_0](u_n^{**}) \lesssim 	u_n^{-3+\eta_T}+ \ep  (E_{p-1}[T\psi_0](u_n^{*}))^{\frac{1}{2}} u_n^{-\frac{1}{2}+ \eta_T(\ep)}.
		\end{equation}  But then, by the boundedness result of Corollary~\ref{energy.low.cor}, we in fact have 
			\begin{equation}
		\sup_{ u\in [u_n^{*},u_{n+1}]}	E_{p-1}[T\psi_0](u) \lesssim 	u_n^{-3+\eta_T}+ \ep  (E_{p-1}[T\psi_0](u_n^{*}))^{\frac{1}{2}} u_n^{-\frac{3}{2}+ \eta_T(\ep)},
		\end{equation} hence, by the Cauchy-Schwarz inequality 
			\begin{equation}
			\sup_{ u\in [u_n^{*},u_{n+1}]}	E_{p-1}[T\psi_0](u) \lesssim 	u_n^{-3+2\eta_T}.
		\end{equation} Hence,  for all $u\geq u_0$
			\begin{equation}
			E_{p-1}[T\psi_0](u) \lesssim 	u^{-3+2\eta_T}.
		\end{equation} Finally we conclude with the proof of \eqref{T.est2}, using Lemma~\ref{lemma.hierarchy}.

	\end{proof}

					\subsection{Finishing the proof of point-wise decay}
				We now turn to the completion of the proof of Theorem~\ref{main.thm}, where \eqref{pointwise.est} remains the only estimate to be established. We will assume the extra assumptions \eqref{boundedness.T}, \eqref{ILED.T}, \eqref{H4} hold throughout this section. First, the point-wise decay of $\phi_{\geq 1}$ is obtained by Corollary~\ref{cor.pointwise.ang} so in the rest of the proof we focus on $\phi_0$, the spherical average.	 We start by establishing the optimal point-wise decay result of $\phi_0$ on the constant $r$-curve $r=R$.		
 
 \begin{lemma}\label{lemma.phi0.sharp.R} The following estimate holds:
 \begin{equation}\label{phi0.sharp.R}
	|\phi_0|(u,v_R(u))\lesssim  u^{-2+O(\sqrt{\ep})}.
\end{equation}
 \end{lemma}
 
 \begin{proof}
We use \eqref{ILED.T} repeating the argument of Lemma~\ref{lemma.average} we find that, combining with Proposition~\ref{T.prop} and Proposition~\ref{prop.high} \begin{equation}\begin{split}&	\int_u^{+\infty} |\phi_0|^2(u',v_R(u')) du' \lesssim  E[\phi_0](u) \lesssim u^{-p_{high} + 2\delta(\ep)}\lesssim u^{-3 + O(\sqrt{\ep})},\\  &	\int_u^{+\infty} |T\phi_0|^2(u',v_R(u')) du' \lesssim  E[T\phi_0](u) \lesssim u^{-2-p_T + \eta_T}\lesssim u^{-5 + O(\sqrt{\ep})}. \end{split}\end{equation}Now, note that\begin{equation}
	|\phi_0|^2(u,v_R(u))\lesssim \bigl| \int_{u}^{+\infty} \phi_0 \cdot T\phi_0(u',v_R(u')) du'\bigr|\ls \left[	\int_u^{+\infty} |\phi_0|^2(u',v_R(u')) du'\right]^{\frac{1}{2}} \left[	\int_u^{+\infty} |T\phi_0|^2(u',v_R(u')) du'\right]^{\frac{1}{2}},
\end{equation} which gives \eqref{phi0.sharp.R}.
 \end{proof}
 Next, we prove similar bounds for $\rd \phi$ on the constant $r$ curve $r=R$.

 
 \begin{lemma}\label{lemma.dphi.pointwise.notsharp} The following estimate holds:
\begin{equation}\label{dphi.pointwise.notsharp}
\sum_{|\alpha|= 1}	|\rd_{\alpha}\phi_0|^2(u,v_R(u))\lesssim  u^{-4+O(\sqrt{\ep})}.
\end{equation} Moreover, there exists $C>0$ independent of $\ep$, such that for all $\eta\in (-1,1)$ such that $|\eta|> C \ep$, \begin{equation} \label{dphi.pointwise.notsharp2}
\int_{v_R(u)}^{+\infty}\left( r^{-1-\eta}|\rd_v \psi_0|^2+ r^{1-\eta}|\rd^2_{v} \psi_0|^2 \right) dv' \lesssim   u^{-4+O(\sqrt{\ep})+\eta}.
\end{equation}
 \end{lemma}
 
 \begin{proof}
 	We take $\eta$ with $0<|\eta|\leq 1$, and write, invoking Proposition~\ref{Hardy.prop} (Hardy inequality) \begin{equation}\label{int}\begin{split}
 	&	\int_{v_R(u)}^{+\infty} r^{-1-\eta} |\rd_v \psi_0|^2(u,v') dv' \lesssim \eta^{-1} \psiR+  \eta^{-2}	\int_{v_R(u)}^{+\infty} r^{1-\eta} |\rd_{v}^2 \psi_0|^2(u,v') dv'\\ & \lesssim |\eta|^{-1}\left( u^{-4+O(\sqrt{\ep})}+	\int_{v_R(u)}^{+\infty} r^{1-\eta} |\rd_{v}^2 \psi_0|^2(u,v') dv'\right) .
 		\end{split}
 	\end{equation}
 	 Next, we write, using \eqref{T.psi.eq} \begin{equation}
 		\int_{v_R(u)}^{+\infty}  r^{1-\eta}|\rd_{v}^2 \psi_0|^2(u,v' ) dv'\lesssim \int_{v_R(u)}^{+\infty}  r^{1-\eta}\left(| \rd_v T \psi_0|^2+\ep^2 r^{-4}(| \psi_0|^2 +| \rd_u\psi_0|^2+r^2| \rd_v\psi_0|^2\right)(u,v' ) dv'.
 	\end{equation}  Then, applying Proposition~\ref{Hardy.prop} (Hardy inequality) to the  $|\psi_0|^2 $ and $| \rd_u\psi_0|^2$ terms,  we end up with 
 	\begin{equation}
 		\int_{v_R(u)}^{+\infty}  r^{1-\eta}|\rd_{v}^2 \psi_0|^2(u,v' ) dv'\lesssim \int_{v_R(u)}^{+\infty}  \left(r^{1-\eta}| \rd_v T \psi_0|^2+\ep^2 r^{-1-\eta}| \rd_v\psi_0|^2\right)(u,v' ) dv' + \ep^2 \left( \psiR+ \dpsiR\right).
 	\end{equation}
 	Now, we apply Proposition~\ref{T.prop} and Lemma~\ref{lemma.phi0.sharp.R} to obtain (denoting $x_+$ the positive part of $x$) 	\begin{equation}\label{int2}
 		\int_{v_R(u)}^{+\infty}  r^{1-\eta}|\rd_{v}^2 \psi_0|^2(u,v' ) dv'\lesssim  u^{-4+[-\eta]_+ +O(\sqrt{\ep})} +\ep^2\int_{v_R(u)}^{+\infty} r^{-1-\eta}| \rd_v\psi_0|^2(u,v' ) dv' + \ep^2   \dpsiR.
 	\end{equation}
 	Combining \eqref{int} and \eqref{int2} gives
 	\begin{equation}\label{int3}
 		\int_{v_R(u)}^{+\infty} r^{-1-\eta} |\rd_v \psi_0|^2(u,v') dv'\lesssim  u^{-4+[-\eta]_++O(\sqrt{\ep})} +\eta^{-2}\ep^2\int_{v_R(u)}^{+\infty} r^{-1-\eta}| \rd_v\psi_0|^2(u,v' ) dv' + \eta^{-2}\ep^2   \dpsiR.
 	\end{equation} 
 	Now, choosing $\eta^{-2} \ep^2=\delta$ small enough (this requires the claimed condition on $\eta$ that $|\eta| > C \ep$), we get 	\begin{equation}\label{int4}
 		\int_{v_R(u)}^{+\infty} r^{-1-\eta} |\rd_v \psi_0|^2(u,v') dv'\lesssim  u^{-4+[-\eta]_++O(\sqrt{\ep})} +\delta   \dpsiR,
 	\end{equation} as well as 
 	\begin{equation}\label{int5}
 		\int_{v_R(u)}^{+\infty} r^{1-\eta} |\rd_v \psi_0|^2(u,v') dv'\lesssim  u^{-4+[-\eta]_++O(\sqrt{\ep})} +\ep^2   \dpsiR.
 		\end{equation}
 	

 To close the estimate, we  	then try to bound $\dpsiR$. We write, for any $v= v_R(u)$, and taking advantage of the fact that $\underset{v\rightarrow +\infty}{\lim} \rd_v \psi_0(u,v)=0$ and Proposition~\ref{Hardy.prop} (Hardy inequality) and Lemma~\ref{lemma.phi0.sharp.R} \begin{equation}\label{int8}\begin{split}
 	&	\dpsiR\lesssim \int_{v_R(u)}^{+\infty} |\psi_0\ \rd_v \psi_0|(u,v' ) dv'\ls [\int_{v_R(u)}^{+\infty} r^{-1-\eta'} |\rd_v \psi_0|^2(u,v' ) dv']^{\frac{1}{2}}[\int_{v_R(u)}^{+\infty}  r^{1+\eta'} |\rd_v^2 \psi_0|^2(u,v' ) dv']^{\frac{1}{2}}\\ &\ls  [u^{-4+|\eta'|+O(\sqrt{\ep})} +\delta   \dpsiR]^{\frac{1}{2}} [u^{-4+|\eta'|+O(\sqrt{\ep})} +\ep^2   \dpsiR]^{\frac{1}{2}}\\ & \ls u^{-4+O(\sqrt{\ep})} +\delta   \dpsiR,
 		\end{split}
 	\end{equation} where in the last two lines, we have applied \eqref{int4} with $\eta = \eta'$  and \eqref{int5} with $\eta = -\eta'$, respectively and chosen $\eta'=O(\sqrt{\ep})$. Choosing $\eta$ so that $\delta$ is sufficiently small, we have established \eqref{dphi.pointwise.notsharp}  with $\alpha=v$, i.e., \begin{equation}\label{int6}
\dpsiR\lesssim  u^{-4+O(\sqrt{\ep})},
 	\end{equation}  and thus, combining with \eqref{int4}, \eqref{int5} also gives  \eqref{dphi.pointwise.notsharp2}. To show \eqref{dphi.pointwise.notsharp}  with $\alpha=u$,
we invoke  Proposition~\ref{energy.decay.prop} applied to $T\phi_0$ and Proposition~\ref{T.prop} to obtain \begin{equation}\label{Tphi.pointwise} 	 |T\phi_0|^2(u,v_R(u)) \lesssim u^{-4+O(\sqrt{\ep})},	 \end{equation} which gives \eqref{dphi.pointwise.notsharp}  with $\alpha=u$ after combining with and noticing that $|\rd_u \phi_0|\ls |\rd_v \phi_0|+|T \phi_0|$.
\end{proof}

 The proof of the following proposition will conclude the proof of \eqref{pointwise.est2}, in view of the above remarks. \begin{prop}
 	The following estimate holds for all $u\geq u_0$, $v\geq v_R(u)$ \begin{equation}\label{pointwise.est2zero}
|\phi_0|(u,v) \lesssim u^{-2+ O(\sqrt{\ep})}.
 	\end{equation}
 \end{prop}
\begin{proof}

To complete the proof of \eqref{pointwise.est2zero}, we propagate \eqref{phi0.sharp.R} of Lemma~\ref{lemma.phi0.sharp.R} in the region where $v \geq v_R(u)$.

We start writing, using Proposition~\ref{Hardy.prop} (Hardy inequality),  Lemma~\ref{lemma.phi0.sharp.R} and Lemma~\ref{lemma.dphi.pointwise.notsharp}: for all $u\geq u_0$, $v\geq v_R(u)$, and any $\eta\in (0,1)$ \begin{equation}\begin{split}\label{sharp.pointwise.R}
&	|\phi_0|^2(u,v) \lesssim 	|\phi_0|^2(u,v_R(u)) +  \int_{v_R(u)}^{v}	r^{-2}|\psi_0||\rd_v\psi_0|(u,v')dv'\\ & \lesssim 	|\phi_0|^2(u,v_R(u)) + [ \int_{v_R(u)}^{v}	r^{-3-\eta}|\psi_0|^2(u,v')dv']^{\frac{1}{2}}[ \int_{v_R(u)}^{v}	r^{-1+\eta}|\rd_v \psi_0|^2(u,v')dv']^{\frac{1}{2}} \\ &\lesssim u^{-4+O(\sqrt{\ep})} +[ \int_{v_R(u)}^{+\infty}	r^{-1-\eta}|\rd_v \psi_0|^2(u,v')dv'+ u^{-4+O(\sqrt{\ep})}]^{\frac{1}{2}}[ \int_{v_R(u)}^{+\infty}	r^{1+\eta}|\rd_{v}^2 \psi_0|^2(u,v')dv'+u^{-4+O(\sqrt{\ep})}]^{\frac{1}{2}}\\ &   u^{-4+O(\sqrt{\ep})} +[ u^{-4+O(\sqrt{\ep})-\eta}+ u^{-4+O(\sqrt{\ep})}]^{\frac{1}{2}}[ u^{-4+O(\sqrt{\ep})+\eta}+u^{-4+O(\sqrt{\ep})}]^{\frac{1}{2}}\ls  u^{-4+O(\sqrt{\ep})},
\end{split}
\end{equation} where, in the last line, we have applied Lemma~\ref{lemma.dphi.pointwise.notsharp} with $\eta  = \pm O(\sqrt{\ep})$. This concludes the proof.

\end{proof}


\bibliographystyle{alpha}
\bibliography{bibliography.bib}

\end{document}